 \documentclass[10pt, reqno]{amsart}
\usepackage{a4,amsthm, amscd}
\usepackage{amssymb} 
\usepackage{a4, amsthm} 
 
\input xy
\xyoption{all}

\theoremstyle{plain}
\newtheorem{theorem}{Theorem}[section]
\newtheorem{lemma}[theorem]{Lemma}
\newtheorem{prop}[theorem]{Proposition}
\newtheorem{cor}[theorem]{Corollary}

\newtheorem{defi}[theorem]{Definition}
\theoremstyle{definition}

\newtheorem{remark}[theorem]{Remark} 

\font\ninerm=cmr9  \font\eightrm=cmr8  \font\sixrm=cmr6
\font\ninei=cmmi9  \font\eighti=cmmi8  \font\sixi=cmmi6
\font\ninesy=cmsy9 \font\eightsy=cmsy8 \font\sixsy=cmsy6
\font\ninebf=cmbx9 \font\eightbf=cmbx8 \font\sixbf=cmbx6
\font\nineit=cmti9 \font\eightit=cmti8 
\font\ninett=cmtt9 \font\eighttt=cmtt8 
\font\ninesl=cmsl9 \font\eightsl=cmsl8

\font\twelverm=cmr12 at 15pt
\font\twelvei=cmmi12 at 15pt
\font\twelvesy=cmsy10 at 15pt
\font\twelvebf=cmbx12 at 15pt
\font\twelveit=cmti12 at 15pt
\font\twelvett=cmtt12 at 15pt
\font\twelvesl=cmsl12 at 15pt
\font\twelvegoth=eufm10 at 15pt

\font\tengoth=eufm10  \font\ninegoth=eufm9
\font\eightgoth=eufm8 \font\sevengoth=eufm7 
\font\sixgoth=eufm6   \font\fivegoth=eufm5
\newfam\gothfam \def\goth{\fam\gothfam\tengoth} 
\textfont\gothfam=\tengoth
\scriptfont\gothfam=\sevengoth 
\scriptscriptfont\gothfam=\fivegoth

\catcode`@=11
\newskip\ttglue

\def\tenpoint{\def\rm{\fam0\tenrm}
  \textfont0=\tenrm \scriptfont0=\sevenrm
  \scriptscriptfont0\fiverm
  \textfont1=\teni \scriptfont1=\seveni
  \scriptscriptfont1\fivei 
  \textfont2=\tensy \scriptfont2=\sevensy
  \scriptscriptfont2\fivesy 
  \textfont3=\tenex \scriptfont3=\tenex
  \scriptscriptfont3\tenex 
  \textfont\itfam=\tenit\def\it{\fam\itfam\tenit}%
  \textfont\slfam=\tensl\def\sl{\fam\slfam\tensl}%
  \textfont\ttfam=\tentt\def\tt{\fam\ttfam\tentt}%
  \textfont\gothfam=\tengoth\scriptfont\gothfam=\sevengoth 
  \scriptscriptfont\gothfam=\fivegoth
  \def\goth{\fam\gothfam\tengoth}
  \textfont\bffam=\tenbf\scriptfont\bffam=\sevenbf
  \scriptscriptfont\bffam=\fivebf
  \def\bf{\fam\bffam\tenbf}%
  \tt\ttglue=.5em plus.25em minus.15em
  \normalbaselineskip=12pt \setbox\strutbox\hbox{\vrule
  height8.5pt depth3.5pt width0pt}%
  \let\big=\tenbig\normalbaselines\rm}

\def\ninepoint{\def\rm{\fam0\ninerm}
  \textfont0=\ninerm \scriptfont0=\sixrm
  \scriptscriptfont0\fiverm
  \textfont1=\ninei \scriptfont1=\sixi
  \scriptscriptfont1\fivei 
  \textfont2=\ninesy \scriptfont2=\sixsy
  \scriptscriptfont2\fivesy 
  \textfont3=\tenex \scriptfont3=\tenex
  \scriptscriptfont3\tenex 
  \textfont\itfam=\nineit\def\it{\fam\itfam\nineit}%
  \textfont\slfam=\ninesl\def\sl{\fam\slfam\ninesl}%
  \textfont\ttfam=\ninett\def\tt{\fam\ttfam\ninett}%
  \textfont\gothfam=\ninegoth\scriptfont\gothfam=\sixgoth 
  \scriptscriptfont\gothfam=\fivegoth
  \def\goth{\fam\gothfam\tengoth}
  \textfont\bffam=\ninebf\scriptfont\bffam=\sixbf
  \scriptscriptfont\bffam=\fivebf
  \def\bf{\fam\bffam\ninebf}%
  \tt\ttglue=.5em plus.25em minus.15em
  \normalbaselineskip=11pt \setbox\strutbox\hbox{\vrule
  height8pt depth3pt width0pt}%
  \let\big=\ninebig\normalbaselines\rm}

\def\eightpoint{\def\rm{\fam0\eightrm}
  \textfont0=\eightrm \scriptfont0=\sixrm
  \scriptscriptfont0\fiverm
  \textfont1=\eighti \scriptfont1=\sixi
  \scriptscriptfont1\fivei 
  \textfont2=\eightsy \scriptfont2=\sixsy
  \scriptscriptfont2\fivesy 
  \textfont3=\tenex \scriptfont3=\tenex
  \scriptscriptfont3\tenex 
  \textfont\itfam=\eightit\def\it{\fam\itfam\eightit}%
  \textfont\slfam=\eightsl\def\sl{\fam\slfam\eightsl}%
  \textfont\ttfam=\eighttt\def\tt{\fam\ttfam\eighttt}%
  \textfont\gothfam=\eightgoth\scriptfont\gothfam=\sixgoth 
  \scriptscriptfont\gothfam=\fivegoth
  \def\goth{\fam\gothfam\tengoth}
  \textfont\bffam=\eightbf\scriptfont\bffam=\sixbf
  \scriptscriptfont\bffam=\fivebf
  \def\bf{\fam\bffam\eightbf}%
  \tt\ttglue=.5em plus.25em minus.15em
  \normalbaselineskip=9pt \setbox\strutbox\hbox{\vrule
  height7pt depth2pt width0pt}%
  \let\big=\eightbig\normalbaselines\rm}

\def\twelvepoint{\def\rm{\fam0\twelverm}
  \textfont0=\twelverm\scriptfont0=\tenrm
  \scriptscriptfont0\sevenrm
  \textfont1=\twelvei\scriptfont1=\teni
  \scriptscriptfont1\seveni 
  \textfont2=\twelvesy\scriptfont2=\tensy
  \scriptscriptfont2\sevensy 
   \textfont\itfam=\twelveit\def\it{\fam\itfam\twelveit}%
  \textfont\slfam=\twelvesl\def\sl{\fam\slfam\twelvesl}%
  \textfont\ttfam=\twelvett\def\tt{\fam\ttfam\twelvett}%
  \textfont\gothfam=\twelvegoth\scriptfont\gothfam=\ninegoth 
  \scriptscriptfont\gothfam=\sevengoth
  \def\goth{\fam\gothfam\twelvegoth}
  \textfont\bffam=\twelvebf\scriptfont\bffam=\ninebf
  \scriptscriptfont\bffam=\sevenbf
  \def\bf{\fam\bffam\twelvebf}%
  \tt\ttglue=.5em plus.25em minus.15em
  \normalbaselineskip=12pt \setbox\strutbox\hbox{\vrule
  height9pt depth4pt width0pt}%
  \let\big=\twelvebig\normalbaselines\rm}

\def\n{{\goth n}}  
\def\k{{\goth k}}

\def\g{{\goth g}}
\def\t{{\goth t}}
\def\l{{\goth l}}
\def\m{{\goth m}}
\def\p{{\goth p}}

\def\s{{\goth s}}
\def\uu{{\goth u}}

\def\ad{{\rm ad}}
\def\Ad{{\rm Ad}}
\def\a{{\mathfrak{a}}}

\def\s{{\mathfrak{s}}}

\def\m{{\mathfrak{m}}}   
\def\n{{\mathfrak{n}}}

\def\p{{\mathfrak{p}}}     
\def\k{{\mathfrak{k}}}
\def\t{{\mathfrak{t}}}

\def\g{{\mathfrak{g}}}
\def\l{{\mathfrak{l}}}

\def\Z{{\mathbb{Z}}}

\def\C{{\mathbb{C}}}
\def\R{{\mathbb{R}}}

\def\N{{\mathbb{N}}}

\def\bbigskip{\bigskip\bigskip}
\def\nbigskip{\bigskip\noindent}

\def\span{{\rm span}}

\def\conv{{\rm Conv}} \def\int{{\rm int}}

\def\buildunder#1#2{\mathrel{\mathop{\kern0pt #2}
\limits_{#1}}}

\def\qq{/\kern-.185em /}

\def\pn{\par\noindent}
\def\sn{\smallskip\noindent}
\def\mn{\medskip\noindent}
\def\bn{\bigskip\noindent}
\def\REM #1{}

\begin{document}

\title[Orbit structure]{Invariant envelopes of holomorphy in  the complexification of a Hermitian symmetric space}

\bbigskip

\author[Geatti]{L. Geatti}
\author[Iannuzzi]{A. Iannuzzi}
\address{Laura Geatti and Andrea Iannuzzi: Dip.~di Matematica,
II Universit\`a di Roma  ``Tor Vergata", Via della Ricerca Scientifica,
I-00133 Roma, Italy.} 
\email{geatti@mat.uniroma2.it,
       iannuzzi@mat.uniroma2.it}




\bigskip
\begin{abstract}
In this paper we investigate   invariant  domains in $\, \Xi^+$,  a distinguished  $\,G$-invariant,  Stein domain  in the complexification  of an irreducible Hermitian symmetric space $\,G/K$. The domain $\,\Xi^+$, recently introduced by Kr\"otz and Opdam,  contains  the crown domain $\,\Xi\,$ and it is maximal with respect to  properness
of the  $\,G$-action. In the tube case, it also contains $\,S^+$, an invariant
Stein domain arising from the compactly causal structure  of a symmetric orbit in the boundary of $\,\Xi$.  We prove that the envelope of holomorphy of an  invariant domain in $\,\Xi^+$, which is contained neither in $\,\Xi\,$ nor in $\,S^+$, is univalent and
coincides with $\,\Xi^+$.  This fact, together with known results concerning
$\,\Xi\,$ and $\,S^+$, proves the univalence of the envelope  of holomorphy of an arbitrary invariant domain  in $\,\Xi^+\,$ and   completes the classification of invariant Stein  domains therein.
\end{abstract}

\maketitle


\bigskip

\section{Introduction}
\bigskip
Let $\,G/K\,$ be a non-compact, irreducible, Riemannian symmetric space. Its
Lie group  complexification  $\,G^\C/K^\C\,$ is a Stein manifold and
left translations by  elements of $\,G\,$ are holomorphic transformations of
$\,G^\C/K^\C$.
In this situation, $G$-invariant domains in $\,G^\C/K^\C\,$ and their
envelopes of holomorphy are natural objects to study.

A first  example is given by
the crown $\,\Xi$, introduced by D. N. Akhiezer and S. G. Gindikin in \cite{AkGi90}.
This Stein invariant domain carries an invariant K\"ahler structure
intrinsically associated with the Riemannian structure  of the symmetric space $G/K$ and,
in many respects,
can be regarded as its canonical  complexification.
In recent years, it  has  been  extensively studied in connection with harmonic analysis on $\,G/K$ (see, e.g \cite{KrSt04}, \cite{KrSt05}).

If $\,G/K\,$ is a Hermitian symmetric space of tube type, two additional
distinguished invariant  Stein domains $\,S^\pm\,$ arise from the compactly casual structure of a pseudo-Riemannian  symmetric space $\,G/H\,$  lying
on the boundary of $\,\Xi$. The  complex geometry  of $\,S^\pm\,$ was
studied by K. H. Neeb in \cite{Nee99}. Inside  the crown $\,\Xi$,
as well as inside
$\,S^\pm\,,$ an invariant domain  can be described via a semisimple abelian  slice,
its envelope of holomorphy  is univalent and Steiness
is characterized by logarithmic  convexity of such a  slice.

One may ask how far  the above results are  from a complete description
of envelopes of holomorphy and 
 a classification of invariant Stein domains  in $\,G^\C/K^\C\,$.
In
\cite{GeIa08}, 
a  univalence result
for $\,G$-equivariant Riemann domains over $\,G^\C/K^\C\,$, and in particular for envelopes of holomorphy,  was proven  in the rank-one case.  In addition, the complete  classification of invariant Stein domains was obtained.
From this  result  one sees that, up to finitely many exceptions, all invariant Stein domains  are  contained either  in $\,\Xi\,$ or, in the Hermitian case of tube type,  in $\,S^\pm$.
Moreover, the study  the CR-structure of principal $\,G$-orbits in $\,G^\C/K^\C$ (i.e.\,closed orbits of maximal dimension)
carried out  in \cite{Gea02}, suggests   that  the latter fact  holds true also  in the higher rank case, the exceptions being finitely  many invariant  domains whose boundary entirely consists of non-principal $\,G$-orbits.

Here we focus on the case of $G/K$ irreducible of Hermitian type.
In  this   case,  B. Kr\"otz and E. Opdam
recently singled out  two Stein, invariant domains
$\,\Xi^+\,$ and $\,\Xi^-\,$ in
$\,G^\C/K^\C$, satisfying  $\,\Xi^+ \cap \Xi^- =\Xi\,$ and   maximal with respect
to properness of the $\,G$-action.
The   relevance of the crown $\,\Xi\,$ and of the domains  $\,\Xi^+\,$ and $\,\Xi^-\,$ for the representation theory of $\,G\,$  was underlined
in Theorem 1.1 in \cite{Kro08}.
Since  $\,\Xi^+\,$ and $\,\Xi^-\,$ are $\,G$-equivariantly anti-biholomorphic, in the sequel we simply
refer  to $\,\Xi^+$.
If $\,G/K\,$ is Hermitian of tube type,  then $\,\Xi^+$  contains both the
crown $\,\Xi\,$ and  the domain $\,S^+$ (\cite{GeIa13}, Prop. 8.7). Moreover,
for $\, r:={\rm rank}(G/K)>1$, the complement of $\,\Xi\cup S^+\,$ in
$\,\Xi^+\,$ has non-empty interior.
  Our main  result is as follows.

\bn
{\bf Theorem.}
{\sl   Let $\,G/K\,$ be an irreducible Hermitian symmetric space.
Given a $\,G$-invariant domain  $\,D \,$ in $\,\Xi^+$, denote by $\,\widehat D\,$
its envelope of holomorphy.
\pn
{\rm (i)}   Assume $\,G/K\,$ is of tube type. If $\,D\,$ is not contained in $\,\Xi\,$ nor in
$\,S^+$, then $\,\widehat D\,$ is univalent and coincides with $\,\Xi^+\,$.
\pn
{\rm (ii)}  Assume $\,G/K\,$ is not of tube type.  If $\,D\,$ is not contained in $\,\Xi\,$, then $\,\widehat D\,$ is univalent and coincides with $\,\Xi^+$.   }

\bigskip
The envelopes of holomorphy of invariant domains in $\,\Xi\,$ or $\,S^+\,$
are known to be univalent and their Steiness is characterized  in terms of the aformentioned semisimple
abelian slices.
Hence, the above theorem implies the univalence of  the envelope  of holomorphy
of an arbitrary  invariant domain  in $\,\Xi^+\,$ and yields the following
classification.

\mn
{\bf Corollary.}
{\sl
Let $\,G/K\,$ be an irreducible  Hermitian symmetric space and let $\,D\,$ be a Stein $\,G$-invariant proper domain in $\,\Xi^+$.
\pn
{\rm (i)}   If $\,G/K\,$ is of tube type, then  either
$\,D\subseteq \Xi\,$  or  $\, D\subseteq S^+$.
\pn
{\rm (ii)}  If $\,G/K\,$ is not of tube type, then $\,D\subseteq \Xi$. }

\bn

The    theorem is proved by  showing  that the natural
$\,G$-equivariant embedding
$\,f: D \to \widehat D\,$ admits a  holomorphic extension  $\,\hat f\colon \Xi^+\to \widehat D$ to the whole $\Xi^+$.
For this purpose, we use the unipotent, abelian slice of $\Xi^+$
pointed out by B.  Kr\"otz and E. Opdam in \cite{KrOp08}.
Namely, one has
$$\Xi^+=G\cdot \Sigma\, ,$$
where
$\,\Sigma :=\exp i {\Lambda_r^\llcorner} \cdot x_0\, \,$
and   $\,\Lambda_r^\llcorner\,$ is a
closed hyperoctant
in an $\,r$-dimensional, nilpotent, abelian subalgebra of $\,Lie(G)$.
This sets
a one-to-one correspondence
$$D \to \Sigma_D:= D \cap \Sigma $$
between
$\,G$-invariant domains in $\,\Xi^+$ and
domains in $\,\Sigma\,$ which are invariant under the action of an appropriate
Weyl group  (see Sect.\,3).

Then a key ingredient is given by Proposition 4.7,
which implies  that
a continuous extension of $\,f |_{\Sigma_D}\,$ to a domain
$\,\widetilde \Sigma\,$ in $\,\Sigma\,$
induces
a $\,G$-equivariant, holomorphic extension of $\,f\,$ on $\,G \cdot
\widetilde \Sigma\,$ provided that  certain  compatibility conditions
are satisfied.
In order to obtain $\,\hat f$, we therefore construct a continuous
extension of  $\,f|_{\Sigma_D}\,$ to $\,\Sigma\,$ satisfying
such compatibility conditions.

This is done in a finite number of steps.
At each step  we  extend $\,f|_{\Sigma_D}\,$ to a larger  domain $\,\widetilde \Sigma \subset \Sigma\,$ properly containing $\,\Sigma_D$.
Such  extensions are obtained  by  equivariantly embedding  in $\,G^\C/K^\C\,$
various lower dimensional complex homogenous manifolds $\,L^\C/H^\C$,  all of whose $\,L$-invariant domains have univalent and well understood envelopes of holomorphy.
The embedding of each  space $\,L^\C/H^\C\,$ is carefully chosen,  so that it intersects $\,D\,$ in some $\,L$-invariant domain $\,T\subset L^\C/H^\C$.
As a consequence,
the holomorphic map
$\,f|_T \colon T\to \widehat D\,$ extends $\,L$-equivariantly to
$\,\widehat T\to \widehat D\,$ and, in particular, yields
a real-analytic extension of $\,f|_{\Sigma_D}\,$ along the submanifold
$\,\widehat T \cap \Sigma$.   It turns out that for some choices of $\,L^\C/H^\C$ 
the intersection  $\,\widehat T \cap \Sigma\,$ is not 
open in $\,\Sigma$.  In these  cases,  an extension of $\,f|_{\Sigma_D}\,$ to an open domain $\,\widetilde \Sigma\subset \Sigma\,$ is obtained by embedding in $D$  a continuous family of  copies of $\,T$.

The real homogenous manifolds $\,L/H\,$ which play a role in our situation
are: real $\,r$-dimensional vector spaces acted on by $\,(\R^r,+)$, the Euclidean plane acted on by its isometry group,  and irreducible rank-one, Hermitian symmetric
spaces,  of both tube-type and non-tube type. In the latter
case, the univalence results on equivariant
Riemann domains obtained  in \cite{GeIa08} are  crucial.
The above  strategy was inspired by the work  of
K. H. Neeb on bi-invariant domains in the complexification of a Hermitian
semisimple Lie group (\cite{Nee98}).

The paper is organized as follows.
In section 2, we set up the notation and recall some preliminary facts which are needed in the paper. In section 3, we recall the unipotent paramentrization of
$\,\Xi^+\,$
and of its $\,G$-invariant subdomains. In section 4, we recall  some basic facts about envelopes of holomorphy and develope the main tools used in the proof of the main theorem.  In section 5 we prove the main theorem.

\bn

\bn

\section{Preliminaries}
 
\bigskip
Let $G/K$ be an irreducible Hermitian symmetric space of the non-compact type. 
 We may assume $G$ to be a connected,  non-compact,  real simple Lie group contained in its   simple,  simply connected  universal complexification~$G^\C$, and $K$  to be a maximal compact subgroup of $G$. Denote by $\g$ and $\k$ the Lie algebras of $G$ and $K$, respectively. 
Denote by  $\theta$ both  the Cartan involution of $G$ with respect to $K$ and the associated involution of $\g$. Let $\g=\k\oplus \p$ be the corresponding Cartan decomposition.  
 Let  $\a$ be a maximal abelian subspace in $\p$. The {\it rank} of $G/K$ is by definition  $r=\dim \a$.  The adjoint action of $\a$   decomposes $\g$  as
 $$\g= \a\oplus Z_{\k}(\a)\oplus\bigoplus_{\alpha\in\Delta(\g,\a)}\g^\alpha,$$ where $Z_{\k}(\a)$ is the centralizer of $\a$ in $\k$, the joint eigenspace  $\g^\alpha=\{X\in\g~|~ [H,X]=\alpha(H)X, {\rm \ for\  every \ } H\in\a\}$ is the $\alpha$-restricted root space and  $\Delta(\g,\a)$ consists of those  $\alpha\in\a^*$ for which $\g^\alpha\not=\{0\} $.
A set of simple roots $\Pi_\a$ in $\Delta(\g,\a)$ uniquely determines a set of  positive restricted roots $\Delta^+(\g,\a)$ and   an Iwasawa decomposition of $\g$  
$$\g=\k\oplus \a\oplus \n,\qquad\hbox{where}\quad \n=\bigoplus_{\alpha\in\Delta^+(\g,\a)}\g^\alpha\,.
$$
The restricted root system of a Lie algebra $\g$  of Hermitian type is either of type $C_r$ (if $G/K$ is of tube type) or of type $BC_r$ (if $G/K$ is not of tube type) (cf. \cite{Moo64}), i.e.  there exists a basis $\{e_1,\ldots,e_r\}$ of $\a^*$ for which   
$$\Delta(\g,\a)=\{\pm 2e_j, ~1\le j\le r,~~\pm e_j\pm e_k,~ 1\le j\not= k\le r\},\quad \hbox{ for type $C_r$}, $$
$$\Delta(\g,\a)=\{\pm e_j,~\pm 2e_j,~1\le j\le r,~~\pm e_j\pm e_k,~~1\le j\not= k\le r\},\quad \hbox{ for type $BC_r$}.$$
Since  $\g$   admits  a compact Cartan subalgebra $\t\subset \k\subset \g$, there exists   a  set of $r$ positive long strongly orthogonal restricted  roots
 $\{\lambda_1,\ldots,\lambda_r\} $ (i.e. such that  $\lambda_j\pm \lambda_k\not\in\Delta(\g,\a)$, for $j\not=k$), which  are restrictions of {\it real}  roots  with respect to a maximally split $\theta$-stable Cartan subalgebra  $\l$  of $\g$ extending $\a$.  

Taking  as simple roots
$
\Pi_\a=\{ e_1-e_2,\ldots, e_{r-1}-e_r,2e_r\}$, for type $C_r$, and  
$\Pi_\a=\{ e_1-e_2,\ldots, e_{r-1}-e_r, e_r\}$, for type $BC_r$, 
one has 
$$
\lambda_1=2e_2,\ldots,\lambda_r=2e_r \,.
$$

Let $Z_0$ be the element in $Z(\k)$ defining the complex structure $J_0=\ad_{Z_0}$ on $G/K$.
 For   $j=1,\ldots,r$, choose $E_j \in \g^{\lambda_j}$ such that the $\s \l (2)$-triple 
$$ \{E_j,~\theta E_j,~A_j:=[\theta E_j,E_j]\}$$
is normalized as follows  
\begin{equation}
\label{TRIPLES}
[A_j,E_j]=2E_j,\quad 
[Z_0,E_j-\theta E_j]=A_j,\quad [Z_0,A_j]=-(E_j-\theta E_j) \,. 
\end{equation}
Then the vectors $\{A_1,\ldots,A_r\} $ form  an orthogonal basis of $\a$ (with respect to the restriction of the Killing form) and  
\begin{equation}
\label{RELATIONS2}
[E_j,E_k]=[E_j,\theta E_k]=0,\quad [A_j,E_k]=\lambda_k(A_j)E_k=0, \quad {\rm for}~j\not=k\, .
\end{equation}
That is, the above $\s \l (2)$-triples commute with each other.
Moreover, under the above choices, the element $Z_0$ is given by
\begin{equation}
\label{CENTER}
 Z_0=S+{1\over 2}\sum T_j,
 \end{equation}
where $T_j=E_j+\theta E_j$ and  $S\in  Z_\k(\a)$ (see Lemma 2.4 in \cite{GeIa13}). 
If $G/K$ is of tube type 
one has  $S=0$.

In the sequel, we denote by $\g_j$ the $\s\l(2)$-triple corresponding to the root $\lambda_j\in\{\lambda_1,\ldots,\lambda_r\}$, and by $G_j$ the corresponding connected subgroup of $G$.
In the non-tube case,  to each  $\lambda_j$   one can also associate a connected, simple, real rank-one Hermitian subgroup 
$ G_j^\bullet$ of $G$.  The  group $G_j^\bullet$ is by definition the connected, 
 $\theta$-stable subgroup of $G$ with Lie algebra 
 $$\g_j^\bullet=\R A_i\oplus \g^{\pm \lambda_j/2}\oplus  \g^{\pm \lambda_j}$$
  isomorphic to $\s\uu(m,1)$,  for some $ m>1$ (see \cite{Kna04}). 

\bigskip
\begin{lemma}\label{COMMUTATIVITY}
Let $G/K$ be an irreducible Hermitian symmetric space, which is not of tube type. Let  $G_j^\bullet$ be the simple real rank-one Hermitian subgroup associated to the root $\lambda_j$, for some $j\in\{1,\ldots,r\}$. Then $G_j^\bullet$ commmutes with  the   subgroups $G_k$, for every $k\not=j$. 
 \end{lemma}
 
\smallskip
\begin{proof} 
By relations  (\ref{RELATIONS2}), one has $[\g_j, \g_k]\equiv 0$, for  $k \not= j$. Futhermore, since $\pm e_j\pm 2 e_k$,  for  $j\not=i$, are not roots in $\Delta(\g,\a)$ and $e_j(A_k)=\delta_{jk}$,  one also has   
$ [\g^{\pm \lambda_j/2},\g_k]\equiv 0 $.
Summarizing,  there is commutativity at Lie algebra level
$$[\g_j^\bullet,\g_k]\equiv 0, \quad \hbox{ for }  k\not= j $$
and likewise at group  level, by connectedness. 
 \end{proof}

\bn


\section{Invariant  subdomains of $\Xi^+$.}

\bigskip
 Let  $G/K$ be an irreducible Hermitian symmetric space of the non-compact type. Its complexification $G^\C/K^\C$ contains a distinguished $G$-invariant Stein subdomain~$\Xi^+$, properly containing  the crown $\Xi$,  and maximal with respect to proper $G$-action.

A description  of the domain $\Xi^+$ was given in \cite{Kro08}, p.286,  and  \cite{KrOp08}, Sect.8,  via its   unipotent  parametrization.
More precisely, fix vectors $E_j\in\g^{\lambda_j}$
normalized as in (\ref{TRIPLES}).  
Then 
$$\Xi^+=G\exp i\bigoplus_{j=1}^r(-1,\infty)E_j\cdot x_0.$$
Define the   
nilpotent abelian
 subalgebras 
 $$\Lambda_r:=\span_\R\{E_1,\ldots,E_r\}  \quad {\rm and} \quad
 \Lambda_r^\C:=\span_\C\{E_1,\ldots,E_r\}$$ 
  of $\n$ and $\n^\C$, respectively. The exponential map of $ G^\C$ 
defines a biholomorphism between $\Lambda_r^\C $
and the  unipotent abelian complex subgroup  $L^\C:=\exp \Lambda_r^\C$.  In particular, it restricts to a diffeomorphism between $\Lambda_r$ and the real unipotent subgroup $L:=\exp \Lambda_r$.
Since   the map  \begin{equation} \label{IOTA}
\iota:\n^\C \to N^\C \cdot x_0,\qquad Z \to \exp Z \cdot x_0,\end{equation}
is a biholomorphism onto its image (cf. Prop. 1.3 in \cite{KrSt04}), so is its restriction 
$\iota:\Lambda_r^\C \to L^\C \cdot x_0$.

\bigskip
\begin{lemma}
\label{STEINTUBE}
The intersection   $\Xi^+\cap L^\C\cdot x_0$ is a closed,   $r$-dimensional,
complex submanifold of $\Xi^+$, which is biholomorphic, via the map $\iota$, to
the Stein tube domain
$\Lambda_r \times i\bigoplus_{j=1}^r(-1,\infty)E_j$ of $\Lambda_r^\C$.
\end{lemma}

\smallskip
\begin{proof} By a result of Rosenlicht (\cite{Ros61}, Thm.\,2),  the orbits of the unipotent subgroup  $L^\C$ in the affine space $G^\C/K^\C$ are closed. In particular 
$L^\C\cdot x_0\cap \Xi^+$ is closed in $\Xi^+$. Now the statement follows from the injectivity of 
  map $\iota$ and the fact that 
  the set $ \{ \,X \in \Lambda_r \ : \ \exp iX \cdot x_0 \in \Xi^+ \, \}$ 
coincides with  
$\bigoplus_{j=1}^r(-1,\infty)E_j$ (see \cite{Kro08}, p. 286).
 \end{proof}

\bigskip
By Lemma 4.1 in  \cite{GeIa13}, the group 
$$W_K(\Lambda_r):=N_K(\Lambda_r)/Z_K(\Lambda_r)\,$$
is a proper subgroup of the Weyl group $N_K(\a)/Z_K(\a)$.
Its action on $\Lambda_r$ is described by the following result.


\bigskip
\begin{lemma}
\label{CENTRNORM}
$($\cite{GeIa13},  Lemma 4.1$)$
The group
$W_K(\Lambda_r)$ acts on  
$\Lambda_r$ by permutations of the basis  elements $\{E_1,\ldots,E_r\}$.
\end{lemma}

\bigskip
One may expect  that the intersection of a $G$-orbit in $\Xi^+$ with the closed
slice $\exp (i \bigoplus_{j=1}^r (-1,+\infty)E_j)\cdot x_0$ is just  a 
 $W_K(\Lambda_r)$-orbit. However, as observed in \cite{GeIa13}, Remark 7.6 
this is not the case. 
Then, when studying the $G$-invariant geometry of $\Xi^+$, it is useful to consider
 a smaller slice  as follows. Consider  the $W_K(\Lambda_r)$-invariant, closed
hyperoctant 
$$\Lambda^{\llcorner}_r:=\span_{\R^{\ge 0}}\{E_1,\ldots,E_r\}\,$$
of $\Lambda_r$,
and  the nilpotent cone in $\g$ given by
$\mathcal N^+ := \Ad_K( \Lambda_r^\llcorner)$.  As suggested 
in  \cite{KrOp08} and \cite{Kro08}, Remark 4.12, the 
following fact  holds true.


\bigskip
\begin{prop}
\label{HOMEOMORPHISM}
$($\cite{GeIa13}, Prop.\,5.7$)$
The $G$-equivariant map 
$$\psi\colon G\times_K{\mathcal N}^+ \to \Xi^+,\qquad [g,X]\to g\exp iX\cdot x_0$$ is a homeomorphism.
\end{prop}

\bigskip
Given  a $G$-invariant domain  $D\subset \Xi^+$, define an 
open subset of $\bigoplus_{j=1}^r(-1,\infty)E_j$ by
$${\mathcal D} := \{ \,X \in \Lambda_r \ : \ \exp iX \cdot x_0 \in D \, \} .$$
 By the definition of $\mathcal D$ and  Proposition \ref{HOMEOMORPHISM}, the domain 
$D$ 
 can  be written as 
$$D=G\exp i{\mathcal D}\cdot x_0=G\exp i{\mathcal D^\llcorner}\cdot x_0,$$ where 
$
{\mathcal D^\llcorner}:= {\mathcal D }\cap  \Lambda_r^{\llcorner} 
$
is a   $W_K(\Lambda_r)$-invariant open subset of $\Lambda_r^{\llcorner}$.
 

\bigskip
\begin{lemma}
\label{WORBIT} $($\cite{GeIa13}, Lemma 7.4$)$
Let $X$ be an element in $\Lambda_r^{\llcorner}$. Then
the $\Ad_K$-orbit of $X$ intersects  $\Lambda_r$ 
in the $W_K(\Lambda_r)$-orbit of $X$ in $ \Lambda_r^{\llcorner}$.
\end{lemma}

\bn
Note that the above result  together with  Proposition \ref{HOMEOMORPHISM} implies that given $X$ in~$\Lambda_r^{\llcorner}$, one has 
$$G \exp i X\cdot x_0 \bigcap \exp i\Lambda_r^{\llcorner} \cdot x_0=
\exp i( W_K(\Lambda_r) \cdot X)  \cdot x_0,$$ 
i.e. every $G$-orbit (not just $ K$-orbit) in $\Xi^+$ intersects 
 the closed slice $\exp i\Lambda^{\llcorner}_r \cdot x_0$
 exactly in a $W_K(\Lambda_r)$-orbit.

Consider the open Weyl chamber
$ (\Lambda_r^{\llcorner})^+ :=
\big \{\sum_{j=1}^rx_jE_j \ : \  
x_1 >  \dots > x_r >0  \big \}$.
By Lemma \ref{CENTRNORM}, its topological closure 
$$\overline{(\Lambda_r^{\llcorner})^+} =
\Big \{ \sum_{j=1}^rx_jE_j,~:~ x_1 \ge  \dots \ge x_r \ge 0
\Big \}\, 
$$
is a perfect slice for the $W_K(\Lambda_r)$-action on $\Lambda_r^{\llcorner}$, 
implying that  $\exp i\overline{(\Lambda_r^{\llcorner})^+} \cdot x_0$ is a perfect slice for the $G$-action on $\Xi^+$. 
It follows that for a $G$-invariant domain $D$ of $ \Xi^+$ one also has
\begin{equation} \label{DPERFECT} D=G\exp i({\mathcal D}^\llcorner)^+\cdot x_0,
\end{equation}  where  the subset 
$ ({\mathcal D}^\llcorner)^+:={\mathcal D}^\llcorner\cap\overline{(\Lambda_r^{\llcorner})^+} $ is
 open in $\overline{(\Lambda_r^{\llcorner})^+}$. 
In particular,
  $({\mathcal D^\llcorner})^+$ is connected if and only if $D$ is connected.

\nbigskip
In the sequel we also need the following fact.


\bigskip
\begin{lemma}
\label{COMPONENTS}
Let $X$ be an element  in $\Lambda_r^{\llcorner}$. Then every connected component of $Z_K(X)$ meets
$ Z_K(\Lambda_r)$.
\end{lemma}

\smallskip
\begin{proof}
Let  $X$ be an arbitrary element in $\Lambda_r^{\llcorner}$.   By (i) of
Lemma 4.1  and Lemma  5.6 in \cite{GeIa13}, one has
$$ Z_K(\Lambda_r)\cong Z_K(\a) \quad\hbox{and}\quad Z_K(X)\cong Z_K(\Psi(X)),$$
where $\Psi(X)=[Z_0,X-\theta X]\in \a$.
Thus in order to prove the lemma, it is sufficient to show that for  an arbitrary element $H\in\a$,
every  connected component  of $Z_K(H)$ meets~$ Z_K(\a)$.

The centralizer $Z_G(H)$ is a $\theta$-stable reductive subgroup of $G$ (see \cite{Kna04}, Prop.\,7.25, p.\,452) of the same rank and real rank as $G$, with maximal compact subgroup $Z_K(H)$. The maximal abelian subspace of 
$Z_\p(H)$ is $\a$ and, as $Z_K(\a)$ is contained in $Z_K(H)$, one has  that
$Z_{Z_K(H)}(\a)=Z_K(\a)$.
Now   Proposition 7.33 in \cite{Kna04}, p.\,457,  applied to the reductive group $Z_G(H)$, states  that
$Z_K(\a)$ meets every connected component of $Z_K(H)$, as desired.
\end{proof}

\bigskip

In \cite{GeIa13}  it was shown that if $G/K$ is of tube type,  then $\Xi^+$  contains another distinguished 
Stein  invariant  domain,  besides the crown $\Xi$. Such domain $S^+$ arises from the compactly causal structure of a pseudo-Riemannian symmetric  $G$-orbit in the 
 boundary of $\Xi$. The domain $S^+$ and its invariant subdomains  were investigated in \cite{Nee99}.
In the unipotent parametrization of $\Xi^+$,  the domains $\Xi$ and  $S^+$ are given as follows (see \cite{KrOp08}, Sect.\,8, \cite{GeIa13}, Prop. 8.7).


 \bn
 \begin{prop}
\label{SOTTODOMINI}
Let $G/K$ be an irreducible  Hermitian symmetric space.   
Inside  $\Xi^+$   the crown  domain $\Xi$ is given by
$$ G \exp i\bigoplus_{j=1}^r[0,1)E_j\cdot x_0\,.$$
If $G/K$ is of tube type,  the  domain  $S^+$ is given by   
$$ G \exp i\bigoplus_{j=1}^r(1,\infty)E_j \cdot x_0 .$$
 \end{prop}   
\bn

\section{Envelopes of holomorphy of invariant domains in $\Xi^+$.} 

In this section we prove 
 some preliminary results 
 supporting  the three basic ingredients of the proof of the main theorem, namely reduction 1,   reduction 2 and rank-one reduction.
A key result is given by Proposition 4.7,
under whose assumptions one obtains 
 $\,G$-equivariant, holomorphic extensions of the embedding~$\,f\colon D\to \widehat D$ to  larger invariant domains  containing $D$.
   
We begin by recalling some general facts about envelopes of holomorphy.
Let $X$ be a Stein manifold and let $D$ be a domain in $X$. By  Rossi's results  \cite{Ros63}, $D$ admits an envelope of holomorphy $\widehat D$.  This means that  there  exist  an open holomorphic embedding  $f\colon D\to \widehat D$ into a Stein manifold $\widehat D$ to which all holomorphic functions on $D$ simultaneously extend. Moreover,  there is a local biholomorphism $q$ such that the   diagram  

\begin{equation} \label{ENVELOPE}
\xymatrix{&\widehat D\ar[d]^q\\
 D\ar[ur]^f \ar[r]^{Id}& X } 
  \end{equation}
 
 \bn
commutes.  

\begin{prop}
\label{UNIVERSALITY} Let $D_1 $ and $D_2 $ be complex  manifolds,  with  envelopes  of holomorphy $f_1\colon D_1\to  \widehat D_1 $ and $f_2\colon D_2\to \widehat D_2 $, respectively. Let $F\colon D_1\to D_2$ be a holomorphic map. Then there exists a unique holomorphic map $\widehat F \colon \widehat D_1\to \widehat D_2$ such that $\widehat F\circ f_1=f_2\circ  F$.
\end{prop}
 
 \pn 
As a consequence of the above proposition, the following facts hold true.
\begin{prop}
\label{TARGET}
 Let $X$ be a Stein manifold and   let $D\subset X$ be a domain with  envelope  of holomorphy $\widehat D$  {\rm (}cf. diagram {\rm (}\ref{ENVELOPE}{\rm )}{\rm )}. 
 \item{{\rm (i)}} Let  $\Omega$ be the smallest Stein domain in $X$ containing $D$. Then $q(\widehat D)$ is contained in $\Omega$. \item{{\rm (ii)}}  Let $\Omega$ be a domain in $X$ containing $D$. Assume there exists a holomorphic map $\hat f\colon \Omega \to \widehat D$ extending $f$.  Then $\widehat \Omega=\widehat D$.  \end{prop}

\bn

If $G$ is a Lie group acting on $X$ by biholomorphisms and the domain $D$ is  $G$-invariant,  then the $G$-action lifts to an action on $\widehat D$ and all the maps in  diagram~(\ref{ENVELOPE}) are  $G$-equivariant.
Coming back to our case, let $$D=G\exp i{\mathcal D}\cdot x_0=G \exp i{\mathcal D}^\llcorner\cdot x_0$$  be a $G$-invariant domain in $\Xi^+$.  Since $\Xi^+$ is Stein, one has   a commutative
diagram

\bn
\begin{equation} 
\xymatrix{&\widehat D\ar[d]^q\\
 D\ar[ur]^f \ar[r]^{Id}& \Xi^+ } 
 \label{envelope}
 \end{equation}

\bn
where all maps are $G$-equivariant. We  prove that under the assumption that $D$  is not entirely contained  in  $\Xi$ nor in    $S^+$ (in the tube case),  the map $f\colon D\to\widehat D$ can be $G$-equivariantly extended to the whole $\Xi^+$. 
We gradually enlarge the domain of definition of $f$ by iterating the following arguments. 

By reduction~1, we show that $f$ can be $G$-equivariantly extended   to a domain $G\exp i\widetilde {\mathcal D}^\llcorner\cdot x_0$  with all the connected components of  $\widetilde {\mathcal D}^\llcorner$ convex (see Prop.\,\ref{REDUCTION1}).  By reduction 2, we show that $f$ can be $G$-equivariantly extended to a domain with $\widetilde {\mathcal D}^\llcorner$ connected (see Prop.\,\ref{REDUCTION2}), and therefore convex.  

The third  basic ingredient  is the rank-one reduction. It is based  on the univalence and the precise description of the envelope of holomorphy of an arbitrary $G$-invariant domain in the complexification of a rank-one Hermitian symmetric space (cf. \cite{GeIa08}).  
Finally,   by applying Proposition  \ref{TARGET}(ii), one  obtains  $\widehat D=\Xi^+$,
The  strategy is similar to the one used  by Neeb in \cite{Nee98}. 

\bn
{\bf The rank-one case.}  
For the reader's convenience we recall the rank-one case,  in the formulation which is needed in this paper.

We begin with  the tube case $G/K=SL(2,\R)/SO(2,\R)\cong SU(1,1)/U(1)$.
Let $\{E, \theta E, A\}$ be the basis of  $\g=\s\l(2,\R)$  defined in (\ref{TRIPLES}).
Then $\Xi^+=G\exp i[0,\infty) E\cdot x_0$  and every $G$-invariant domain
in $\Xi^+$ is of the form $D=G \exp iIE \cdot x_0$, where $I$ is an open,
connected interval in $[0, \infty)$.

The curve $\ell: [0,+\infty ) \to \Xi^+$, given by $t \to \exp it E \cdot x_0$,
starts at $x_0$ and intersects every $G$-orbit in $\Xi^+$ precisely once.
For every $t>0$ the orbit  $G\cdot \ell(t)$ is a real hypersurface in $\Xi^+$.
Denote by  $ T^{CR}_{\ell(t)}(G\cdot \ell(t)):=
T_{\ell(t)}(G\cdot \ell(t))\cap J_{\ell(t)}T_{\ell(t)}(G\cdot \ell(t))$
the complex tangent space to  $G\cdot \ell(t)$
at  $\ell(t)$.
The quadratic Levi form of  $G\cdot \ell(t)$ at  $\ell(t)$ is given by
$$
{\mathcal L}_{\ell(t)}(W,W)=\frac{1}{8}\frac{(t^2-1)}{t}|W|^2  \ \dot \ell(t)\,,
\quad {\rm for} \ W \in   T^{CR}_{\ell(t)}(G\cdot \ell(t)) \,.
$$
The above formula shows that  for every $t \not=1$ the hypersurface $G\cdot \ell(t)$
has non-degenerate definite Levi form. Hence it bounds a
Stein $G$-invariant domain in $\Xi^+$.
Note that the concavities of the hypersurfaces $G\cdot \ell(t)$, for $t<1$, and $G\cdot \ell(t)$, for $t>1$, point in opposite directions. The hypersurface   $G\cdot \ell(1)$
is Levi-flat.
All proper Stein, $G$-invariant subdomains of  $\Xi^+$ are given by
(cf.\,Lemma 8.1 in \cite{GeIa13} and  Ex. 6.2 \cite{GeIa08})
\begin{equation}
\label{CLASSIFICATION}
\begin{aligned}
&G \exp i[0,b)E \cdot x_0, \ \  \quad {\rm for} \quad  0 < b \leq 1 \cr
&G \exp i(a,\infty)E \cdot x_0,  \quad {\rm for} \quad  1 \leq a < \infty\,.
\end{aligned}
\end{equation}
Moreover one has the following description of envelopes of holomorphy in $\Xi^+$.


\bigskip
\begin{prop}
\label{RANKONE1}
Let $G= SL(2,\R)$ and  let $D$ be a  $G$-invariant domain in $  \Xi^+$. Then the
envelope of holomorphy $\widehat D$ of $D$ is univalent and   given as follows.
\smallskip
\item{{\rm (i)}} If  $  D = G \exp i(a, b)E \cdot x_0$ or $  D = G \exp i[0, b)E \cdot x_0$, with $b \leq 1$, then $$\widehat D = G \exp i[0, b)E \cdot x_0;$$

\item{{\rm (ii)}}   If $  D = G \exp i(a, b)E \cdot x_0$ or
$D = G \exp i(a, \infty)E \cdot x_0$ , with $1 \leq a$, then  $$\widehat D = G \exp i(a, \infty) E \cdot x_0;$$

\item{{\rm (iii)}}   If $D$ contains the orbit $G\cdot \ell(1)$, then $\widehat D = \Xi^+$.
\end{prop}
\bn

\begin{proof}
The center $Z$ of $SL(2,\R)$ acts trivially on $D \subset G^\C/K^\C$
and, by the analytic continuation principle, on $\widehat D$.
Thus   
the projection $q: \widehat D \to \Xi^+$ is $PSL(2,\R)$-equivariant
and, by Theorem 7.6 in \cite{GeIa08},   is injective.
Then, by Proposition \ref{TARGET},  the envelope of holomorphy
$\widehat D$ coincides
with the smallest Stein, $G$-invariant domain in $\Xi^+$
containing $D$. The rest of the statement follows from the
classification of all Stein, $G$-invariant domains in $\Xi^+$
given in (\ref{CLASSIFICATION}).
\end{proof}
\bn

Consider now the rank-one, Hermitan symmetric space $G/K=SU(n,1)/U(n)$, for $n>1$, which is not of tube type.
The difference with the previous case lies  in the fact that, for
 $t>1$, the hypersurface  $G\cdot \ell(t)$ has  non-degenerate, indefinite  Levi form.
As a consequence it cannot lie on the boundary of a Stein $G$-invariant domain in
$\Xi^+$.
The hypersurface  $G\cdot  \ell(1)$ has semidefinite Levi form  and
lies on the boundary of the crown domain $\Xi$, which is Stein.
In this case all proper, Stein, $G$-invariant subdomains of $\Xi^+$ are given by
$$G \exp i[0,b),  \quad {\rm for} \quad  0 < b \leq 1 \,,$$
(cf. Lemma 8.1 in \cite{GeIa13}, and  Ex. 6.3 \cite{GeIa08}) 
and   similar  arguments as in Proposition \ref{RANKONE1}
give  the description of the envelopes of holomorphy in this case.


\bigskip
\begin{prop}
\label{RANKONE2}    Let $G= SU(n,1)$ and  let $D$ be a  $G$-invariant domain in $  \Xi^+$. Then the envelope of holomorphy $\widehat D$ is univalent and given as follows.

\smallskip
\item{{\rm (i)}} If  $D = G \exp i(a, b)E \cdot x_0$
or $D = G \exp i[0, b)E \cdot x_0$, with $b\leq 1$, then
$$\widehat D = G \exp i[0, b)E \cdot x_0 \,.$$
\item{{\rm (ii)}}   If $D$ contains an orbit $G\cdot \ell(t)$, for some $t\ge 1$, then $\widehat D = \Xi^+$.
\end{prop}

\bn
{\bf The extension lemma.} The goal of this subsection is to prove the ``extension lemma", which provides sufficient conditions for a continuous lift $f\colon \exp i\mathcal C\cdot x_0\to \widehat D$ to extend  to a $G$-equivariant holomorphic map $\hat f\colon G\exp i\mathcal C\cdot x_0\to \widehat D$. One of the conditions  involves the isotropy subgroups of points $z\in D$ and $f(z)\in \widehat D$.

Since the projection $q\colon \widehat D\to \Xi^+$ is a  $G$-equivariant  local biholomorphism, the isotropy subgroup of $z\in \widehat D$ is the union of connected components  of the isotropy subgroup of $q(z)\in \Xi^+$. 
In addition, since   $f\colon D\to \widehat D$ is a $G$-equivariant biholomorphism onto its image
and $q|_{f(D)} \circ f=Id_D$, there is actually   an identity of isotropy subgroups   
 $G_{z}= G_{q(z)}$,  for all $z \in f(D)$.
In the sequel it will be crucial to have such an identity of isotropy subgroups
for  points lying in suitable submanifolds of $q(\widehat D)$ intersecting
$D$, to which the map $f$  extends holomorphically.


\bigskip
\begin{lemma}
\label{ISOTROPIE}
Let  ${\mathcal C }$ be an open
subset of  $ \Lambda_r^{\llcorner}$ and let  $f:\exp i {\mathcal C }
\cdot x_0 \to \widehat D$ be a continuous map such that $q \circ f=Id$. Assume that  there exists an  open subset ${\mathcal F}$
of ${\mathcal C }$ such that

\item{{\rm (i)} }
$G_{f(\exp iX^\prime\cdot x_0)}=G_{\exp iX^\prime\cdot x_0} $ for all $X^\prime$ in
${\mathcal F}$,

\item{{\rm (ii)}} for every $X\in {\mathcal C }$, there exist an element
$X^\prime \in {\mathcal F}$ such that the segment $\{\,
X^\prime+t(X-X^\prime) \ : \ t\in[0,1] \}$ is contained in
${\mathcal C }$, and a holomorphic extension of $f$ to the
submanifold ${\mathcal S}=\{\, \exp (i(X^\prime + \lambda (X-X^\prime)))\cdot x_0 \ : {\rm Re} \lambda \in [0,1] \, \}$.

\noindent
Then  $G_{f(\exp iX\cdot x_0)}=G_{\exp iX\cdot x_0} $, 
for every $X$ in $\mathcal C $.
\end{lemma}

\begin{proof} Since $q$ is $G$-equivariant and $q \circ f =Id$ on $\exp i{\mathcal C}\cdot x_0$, it is
clear that $G_{f(\exp iX\cdot x_0)} \subset G_{\exp iX\cdot x_0}$
for all $X\in\mathcal{C}$. In order to prove the opposite inclusion,
we consider first   generic elements in $\mathcal{C}$.

By definition, generic elements $X\in \Lambda_r^\llcorner$ are those for which  $Z_K(X)=Z_K(\Lambda_r)$, and by   Lemma 7.3  in \cite{GeIa13}, 
they are dense in $\Lambda_r^\llcorner$.
Let  $X$  be a generic element in $\mathcal{C}$
and let $g$ be an element in $ G_{\exp iX\cdot x_0} = Z_K(\Lambda_r)$.
The fixed point set of $g$ in $\widehat D$
$$Fix(g,\widehat D):=\{z\in \widehat D~|~g\cdot z=z\}$$
 is a complex analytic subset of $\widehat D$.
Let  $X^\prime\in{\mathcal F} $ be an element satisfying condition~(ii) of the lemma.
Since   both ${\mathcal C }$ and ${\mathcal F}$ are open,
 $X^\prime$ can be chosen generic as~well.
Consider the strip  $S:=\{\lambda\in \C \ : \  {\rm Re}\lambda\in\, [0,1]\}$
and define the function 
$$\phi\colon S\to \widehat D,\qquad  \phi(\lambda):=f(\exp i(X^\prime+\lambda(X-X^\prime)\cdot x_0))\,.$$
We  are going to show that the set $$A:=\{\lambda\in S~:~ g \cdot \phi(\lambda)=\phi(\lambda)\}$$ contains the element 1: this implies that
$f(\exp iX \cdot x_0)\in Fix(g,\widehat D)$ and proves the statement for $X$ generic.

Since  both $X$ and $X^\prime$ are generic in $\Lambda_r^\llcorner$, one has that 
 $G_{\exp iX^\prime\cdot x_0} = G_{\exp iX\cdot x_0} =
Z_K(\Lambda_r).$  Therefore  $g \in G_{\exp iX^\prime\cdot x_0}$ and, 
 by condition~(i), it follows that  $f(\exp iX^\prime \cdot x_0)\in Fix(g,\widehat D)$.  
Consequently $0\in A$.  Since ${\mathcal F}$ is open,  there exists
$\varepsilon>0$ such that $[0, \varepsilon) \subset A$.
Let $[0,b)$ be the maximal open interval in $A \cap \R$ containing $0$ and assume by contradiction that $b\not=1$.  Since $A$ is closed,  it follows that $b\in A$ and,
by the definition of $A$, one has that $\phi(b)\in Fix(g,\widehat D)$. Locally, in a neighbourhood $U$ of $\phi(b)$ in $ \widehat D$, the analytic set
$Fix(g,\widehat D)$ is given  as
$$Fix(g,\widehat D)\cap U=\{ z\in U~|~ \psi_1(z)=\ldots=\psi_k(z)=0\},$$ for some $\psi_1,\ldots,\psi_k\in {\mathcal O}(U)$. Thus,
for each $j=1,\ldots r$, the holomorphic function
$$\psi_j\circ \phi\, \colon \phi^{-1}(U) \to \C,\quad \lambda\mapsto \psi_j(f(\exp i(X^\prime+\lambda (X-X^\prime))\cdot x_0))$$
vanishes identically on  $[0,b] $.  Since $\phi^{-1}(U) $ is open in $S$, there exists $\epsilon'>0$ such that  the restriction
$\psi_j\circ \phi_{(b-\epsilon',b+\epsilon')}$ is real analytic and identically zero on $(b-\epsilon',b]$. Hence it is identically zero on the whole interval $(b-\epsilon',b+\epsilon')$, contradicting the maximality of $b$.
Thus $b=1$ and $b \in A$, as claimed. This concludes the case of generic elements in ${\mathcal C}$. 

Consider now a  non-generic element $X\in {\mathcal C}$. Since generic elements form an open dense subset of ${\mathcal C}$, there exists a sequence of generic elements $\{X_n\} \subset
{\mathcal C}$ converging to $X$. Recall that all generic elements in
${\mathcal C }$ have the same isotropy subgroup $Z_K(\Lambda_r)$.
Therefore, by the previous step, one has
$$g\cdot  f(\exp iX_n\cdot x_0)=f(\exp iX_n\cdot x_0),  \quad   {\rm for\ all}\ \  g\in
Z_K(\Lambda_r).$$
Passing to the limit, one obtains that
$g\cdot  f(\exp iX \cdot x_0)=f(\exp iX \cdot x_0)$, for all  $g\in Z_K(\Lambda_r)$.
This fact together with Lemma \ref{COMPONENTS} implies that
$  G_{\exp iX\cdot x_0} \subset G_{f(\exp iX\cdot x_0)}$ for all $X\in {\mathcal C} $,  and concludes the proof of the lemma.
\end{proof}


\bigskip
\begin{lemma}
\label{CONTINUOUS}
Let $D=G\exp i {\mathcal D}^\llcorner\cdot x_0$ be a $G$-invariant domain in $\Xi^+$ and let $X$ be a $G$-space. A $G$-equivariant map $f:D \to X$ is continuous if and only
if its restriction to $\exp i {\mathcal D}^\llcorner \cdot x_0$ is continuous.
\end{lemma}

\smallskip
\begin{proof}
One implication is clear. For the converse,
we first prove that $f$ is continuous on
$K\exp i{\mathcal D}^\llcorner  \cdot x_0 =
\exp i \Ad_K {\mathcal D}^\llcorner  \cdot x_0$.
Consider the identification
$\Ad_K {\mathcal D}^\llcorner  \to  \exp i \Ad_K {\mathcal D}^\llcorner \cdot x_0$ defined by $X \to \exp i X \cdot x_0$ (see Lemma \ref{HOMEOMORPHISM})
and let $X_n \to  X_0$ be a converging sequence in
$\Ad_K {\mathcal D}^\llcorner$.
Choose elements $k_n$ in $K$ such that $\Ad_{k_n}  X_n \in  {\mathcal D}^\llcorner$.
Since $K$  compact, we can assume that the sequence $\{k_n\}_n$ converges to an element  $k_0\in K$ and   that
$\Ad_{k_n}  X_n \to \Ad_{k_0}  X_0$. 

Now observe that  $ {\mathcal D}^\llcorner  = \Lambda_r^\llcorner  \cap
\Ad_K {\mathcal D}^\llcorner $ 
(see Lemma \ref{WORBIT}).
It follows that
$  {\mathcal D}^\llcorner $ is closed in  $\Ad_K {\mathcal D}^\llcorner $, 
implying that 
$\Ad_{k_0}   X_0$ is contained in  ${\mathcal D}^\llcorner $ (and not just in $ \Ad_K {\mathcal D}^\llcorner$).
 Then one has 
$$f(\exp iX_n \cdot x_0)= k_n^{-1} \cdot
f(\exp i (\Ad_{k_n}   X_n) \cdot x_0) \to
k_0^{-1} \cdot  f(\exp i (\Ad_{k_0}   X_0) \cdot x_0 )=$$
$$= f(\exp i X_0 \cdot x_0)\,,$$
which says that $f$ is continuous on $\exp i \Ad_K {\mathcal D}^\llcorner \cdot x_0$, as claimed.

Next, consider the following commutative diagram
$$
\xymatrix{ G \times \Ad_K {\mathcal D}^\llcorner \ar[d]_\pi\ar[dr]^{\tilde f}&\\
D\ar[r]^{f}& \Xi^+  \,,}
$$
where $\pi$ is  the map given  by $(g,X) \to g \exp iX \cdot x_0$
and $\tilde f$  is the lift of $f$ to  $G \times \Ad_K {\mathcal D}^\llcorner $.
As a consequence of Proposition \ref{HOMEOMORPHISM},
the map $f$ is continuous if and only if so is $\tilde f$.
So let $(g_n,X_n) \to (g_0,X_0)$ be a converging sequence in
$G \times \Ad_K {\mathcal D}^\llcorner$. Since  $f$ is continuous on $
\exp i\Ad_K {\mathcal D}^\llcorner \cdot x_0$,
one has
$$\tilde f(g_n,X_n) = f(g_n \exp  iX_n \cdot x_0)=$$
$$g_n\cdot  f(\exp  iX_n \cdot x_0)
\to g_0 \cdot f(\exp  iX_0 \cdot x_0) = f(g_0 \exp  iX_0 \cdot x_0) = \tilde f(g_0,X_0)\,.$$
Thus $\tilde f$ is continuous, implying that $f$ is continuos.
\end{proof}


\bigskip
\begin{lemma} {\bf (Extension lemma)}. 
\label{EXTENSION}
Let ${\mathcal C}$ be an open subset of $\Lambda_r^{\llcorner}$
and let  $f \colon \exp i{\mathcal C }\cdot x_0 \to \widehat D$ be a continuous map
such that  $q\circ f= Id$ and
$G_{\exp iX\cdot x_0} = G_{f(\exp iX\cdot x_0)}$, for every $X\in{\mathcal C}$.
Assume that for every pair $X,X'\in {\mathcal C }$ on
the same $W_K(\Lambda_r)$-orbit
there exists  $n\in N_K(\Lambda_r)$ such that
$$X'=\Ad_n X\quad {\rm and}\quad f(\exp iX'\cdot x_0)=n
\cdot f(\exp iX\cdot x_0) .$$
\pn
Then there exists a unique $G$-equivariant holomorphic map $\hat f \colon G \exp i{\mathcal C }\cdot x_0\to \widehat D$ which  extends $f$.
\end{lemma}

\sn
We point out that  the domain $G \exp i{\mathcal C }\cdot x_0$ coincides with $G \exp i (W_K(\Lambda_r)\cdot {\mathcal C })\cdot x_0$.

\smallskip
\begin{proof} 
If one such $\hat f$ exists, it is uniquely determined by
the relation
$$\hat f(g\exp iX\cdot x_0):=g\cdot f(\exp i X\cdot x_0),\quad \hbox{ for $X\in {\mathcal C}$ and $g\in G$. }$$
First we show that $\hat f$ is well defined.
Assume that $g'\exp i X'\cdot x_0 = g\exp iX\cdot x_0$, for some other $X'\in {\mathcal C}$ and $g'\in G$.
By Proposition \ref{HOMEOMORPHISM},
there exists $k \in K$ such that
$$g'=gk^{-1} \quad \quad  \quad   \quad {\rm and}  \quad  \quad \quad \quad X'=\Ad_kX \,.$$
In addition, by Lemma \ref{WORBIT},  two such  elements $X,~X'\in \Lambda_r^{\llcorner}$,
 lie on the same $W_K(\Lambda_r)$-orbit.
Then, by the compatibility assumption,
there exists $n\in N_K(\Lambda_r)$ such that
$$X'=   \Ad_kX=\Ad_nX
\quad \hbox{and}\quad 
f(\exp iX'\cdot x_0)=n \cdot f(\exp iX \cdot x_0).   $$
In view of the above relations, we obtain
$$g' \cdot f(\exp i X'\cdot x_0)= gk^{-1}n\cdot f(\exp iX \cdot x_0).$$
Now observe that  
  $k^{-1}n\in Z_K(X)$ and  that $Z_K(X)=G_{\exp iX\cdot x_0}=G_{f(\exp iX\cdot x_0)}$, where the first identity follows from Proposition \ref{HOMEOMORPHISM} and the second one from the assumptions.
It follows that
$$g'  \cdot f(\exp i X'\cdot x_0)= g \cdot f(\exp i X \cdot x_0)\,,$$
proving that $\hat f$ is well defined.

Next we show that $\hat f$ is continuous. By Proposition
\ref{HOMEOMORPHISM} and
Lemma \ref{WORBIT},
one has $G \exp i{\mathcal C} \cdot x_0 = G\exp (iW_K(\Lambda_r) \cdot
{\mathcal C}) \cdot x_0$.
Then, by Lemma \ref{CONTINUOUS}, it is sufficient to show that $\hat f$
is continuous  on $\exp (iW_K(\Lambda_r) \cdot
{\mathcal C}) \cdot x_0$, i.e.  on each  set $\exp (i \gamma \cdot {\mathcal C}) \cdot x_0$, 
for  $\gamma$ in $W_K(\Lambda_r)$.
By assumption, $\hat f$ is continuous on $\exp i{\mathcal C}\cdot x_0$.
This settles  the case when $\gamma$ is the neutral element in $W_K(\Lambda_r)$. Otherwise,
write $\gamma = nZ_K(\Lambda_r)$, for some $n  \in N_K(\Lambda_r)$.
Then by the
$G$-equivariance of $\hat f$ one has
$$\hat f(\exp (i \gamma \cdot X) \cdot x_0)=
\hat f(\exp i \Ad_n X \cdot x_0) =
n \cdot\hat f(\exp i X \cdot x_0) \,,$$
for every $X
\in {\mathcal C}$, proving that
$\hat f$ is continuous on $\exp (i \gamma \cdot {\mathcal C}) \cdot x_0$,
as wished.

Finally we show that $ \hat f$ is holomorphic.
Note that  $q \circ \hat f=Id$, since by assumption
 such equality holds true on $\exp {i\mathcal C} \cdot x_0$  and $\hat f$ is $G$-equivariant.
Let $x$ be an element of $G\exp i {\mathcal C} \cdot x_0$ and choose a connected open neighborhood $U$ of $\hat f(x)$ such that the restriction $q|_U:U \to
\hat f (U)$ is a biholomorphism. Then, given a neighborhood $V$
of $x$ such that $\hat f(V) \subset U$, one has $\hat f |_V=
(q|_U)^{-1} \circ Id$, implying that $\hat f$ is holomorphic.
\end{proof}

\bn  
{\bf Reduction 1.}
Let $$D=G\exp i{\mathcal D}\cdot x_0=G\exp i{\mathcal D}^\llcorner\cdot x_0$$
be a $G$-invariant domain in $\Xi^+$.
The first reduction reduces to the case where all connected components of ${\mathcal D}^\llcorner$ are  convex.
It consists of  showing that the map $f$ in diagram (\ref{envelope}) has a $G$-equivariant holomorphic extension to a domain $G\exp i\widetilde{{\mathcal D}}^\llcorner\cdot x_0$,  with $\widetilde{{\mathcal D}}^\llcorner$ a set containing $\mathcal D^\llcorner$, all of whose connected components  are  convex.

We need some preliminary remarks. Recall that $(\widetilde{{\mathcal D}}^\llcorner)^+=\widetilde{{\mathcal D}}^\llcorner\cap (\Lambda_r^\llcorner)^+$ is a perfect slice for $D$ and that it is connected (cf. (\ref{DPERFECT})).

\begin{defi} \label{D1}
Denote by  ${\mathcal D}_\circ$  $($resp. by  $ {\mathcal D}_\circ^\llcorner $  the connected component of ${\mathcal D}$  $($resp. of ${\mathcal D}^\llcorner)$
containing $({\mathcal D}^\llcorner)^+$. \end{defi}

\sn

Note that the set  ${\mathcal D}_\circ$ is open in $\Lambda_r$; 
the set  ${\mathcal D}_\circ^\llcorner$  is open in $\Lambda_r^\llcorner$, while it need not be open
in $\Lambda_r$. Both ${\mathcal D}_\circ$ and ${\mathcal D}_\circ^\llcorner $ need not be $W_K(\Lambda_r)$-invariant.

For  $k\in\{1,\ldots,r-1\}$, denote by  $\gamma_{k  k+1}$ the reflection flipping the $k^{th}$ and the $(k+1)^{th}$ 
coordinates in $\Lambda_r^\llcorner$.   By Lemma \ref{CENTRNORM}  such reflections generate the Weyl group $W_K(\Lambda_r)$.
Denote by $\Gamma^0$ the set of  those $\gamma_{k  k+1}$ for which there exists a non-zero element in
$ Fix(\gamma_{kk+1})\cap ({\mathcal D}^\llcorner)^+$, i.e. whose fixed point  hyperplane intersects $({\mathcal D}^\llcorner)^+$ non-trivially.
Consider the   subgroup of $W_K(\Lambda_r)$
$$W^0:=\langle \{ \, \gamma_{kk+1} \in \Gamma^0  \}\rangle \,, $$
 generated by the elements  of $\Gamma^0$.  

\bigskip
\begin{lemma}
\label{COMPONENTE}
$W^0\cdot ({\mathcal D}^\llcorner)^+ ={\mathcal D}_\circ^\llcorner $.
\end{lemma}

\begin{proof}
Set ${\mathcal C} :=W^0\cdot ({\mathcal D}^\llcorner)^+$.
We first show  that ${\mathcal C}$ is contained in ${\mathcal D}_\circ^\llcorner$.
For this note that $({\mathcal D}^\llcorner)^+\cap \gamma_{kk+1} \cdot ({\mathcal D}^\llcorner)^+  
 \not= \emptyset$, for all  $\gamma_{kk+1} \in \Gamma^0$. Thus $\gamma_{kk+1} \cdot ({\mathcal D}^\llcorner)^+ \subset
{\mathcal D}_\circ^\llcorner$ and  $\gamma_{kk+1} $  stabilizes  
 ${\mathcal D}_\circ^\llcorner$.  Then the whole group $W^0$  stabilizes ${\mathcal D}_\circ^\llcorner$, implying that
${\mathcal C}  \subset {\mathcal D}_\circ^\llcorner$.

Next, we claim that
for $\gamma\in W_K(\Lambda_r)$, one has that 
$\gamma \cdot ({\mathcal D}^\llcorner)^+ \cap {\mathcal C} \not= \emptyset$ if and only if $\gamma \in W^0$.
One implication is clear, since $\gamma \cdot ({\mathcal D}^\llcorner)^+
\subset {\mathcal C}$ if $\gamma \in W^0$.
Conversely, if $\gamma \cdot ({\mathcal D}^\llcorner)^+ \cap {\mathcal C} \not= \emptyset$
there exists $\gamma_1$  in $W^0$ such that
$$\gamma_1 \gamma \cdot ({\mathcal D}^\llcorner)^+ \cap ({\mathcal D}^\llcorner)^+ \not=
\emptyset\,.$$ Since $({\mathcal D}^\llcorner)^+$ is a fundamental region for the
action of $W_K(\Lambda_r)$ on ${\mathcal D}^\llcorner$,  it follows that there exists $X$ in the boundary of $  ({\mathcal D}^\llcorner)^+$  
such that $\gamma_1 \gamma \cdot X = X$. In other words,
$\gamma_1 \gamma $ lies in  the stabilizer subgroup
$W_K(\Lambda_r)_X$ of $X$ in $W_K(\Lambda_r)$. Since $W_K(\Lambda_r)_X$ is generated by the elements
$\gamma_{kk+1}$ in $\Gamma^0 \cap W_K(\Lambda_r)_X$ (see \cite{BrTD85}, Thm.4.1, p.\,202), 
one has that   $\gamma_1\gamma\in W^0$. Then  $\gamma \in W^0$, as claimed.

It follows that  ${\mathcal D}^\llcorner$ is the
union of  the two disjoint subsets
$${\mathcal C} \quad \quad \quad  {\rm and} \quad \quad \quad
\bigcup_{\gamma \in W_K(\Lambda_r) \setminus W^0} \gamma \cdot ({\mathcal D}^\llcorner)^+\ .$$
As $({\mathcal D}^\llcorner)^+$ is closed in
${\mathcal D}^\llcorner$, both subsets are closed in ${\mathcal D}^\llcorner$.
Thus ${\mathcal C}$ must be the union of connected components of ${\mathcal D}^\llcorner$.
Since we already showed that ${\mathcal C} \subset {\mathcal D}_\circ^\llcorner $, it follows that
${\mathcal C} = {\mathcal D}_\circ^\llcorner$, as stated.
\end{proof}


\REM{
\begin{lemma}
\label{CONVEXITY}
If the  $G$-invariant domain in $D$ is Stein, then every connected component of 
${\mathcal D}$ is convex.
\end{lemma}
\smallskip
\begin{proof}
By Lemma \ref{STEINTUBE}, the  intersection  $D\cap L^\C\cdot x_0$ is Stein. In addition, it is biholomorphic, via the map $\iota$,  to  the tube domain 
 $\Lambda_r \times i{\mathcal D}$.
 Then the result follows from Bochner's tube theorem.
\end{proof}  }


\bn
\begin{prop}
\label{REDUCTION1} {\bf (Reduction 1)}
The inclusion $f\colon D\hookrightarrow \widehat D$  extends holomorphically and $G$-equivariantly to the $G$-invariant domain $G \exp i \conv({\mathcal D}_\circ^\llcorner)\cdot x_0 .$
\end{prop}

\smallskip
\begin{proof}
Let  ${\mathcal D}_\circ \subset {\mathcal D}$ be the connected component defined in Definition \ref{D1}.   
By Lemma \ref{STEINTUBE}, the  intersection  $D\cap L^\C\cdot x_0$ is a closed $r$-dimensional $L$-invariant complex  submanifold of $D$,  biholomorphic, via the map $\iota$,  to  the tube domain 
 $\Lambda_r \times i{\mathcal D}$.

By  Bochner's tube theorem, its envelope of holomorphy is univalent
and  given by
$  L\exp i \conv({\mathcal D}_\circ) \cdot x_0\subset \Xi^+$.
Then, by Proposition \ref{UNIVERSALITY},  the map $f$  admits a holomorphic extension  to an $L$-equivariant map $$L\exp i  \conv({\mathcal D}_\circ) \cdot x_0\to \widehat D.$$
Note that the convexification  $\conv({\mathcal D}_\circ)$ contains $\conv({\mathcal D}_\circ^\llcorner)$, which is an open  subset of $\Lambda_r^\llcorner$ and coincides with $\conv({\mathcal D}_\circ) \cap \Lambda_r^\llcorner$.
Moreover, given $X\in \conv({\mathcal D}_\circ^\llcorner)$ and $X^\prime\in   {\mathcal D}_\circ^\llcorner $, the one-dimensional complex manifold
 $$\mathcal S=\{\, \exp (i(X^\prime + \lambda (X-X^\prime)))\cdot x_0 \ : {\rm Re} \lambda \in [0,1] \, \}=$$
$$=  \{\, \exp s(X-X^\prime)\exp (i(X^\prime + t (X-X^\prime)))\cdot x_0 \ : s\in\R,~ t \in [0,1] \, \}$$ 
is contained in $  L\exp i  \conv(\mathcal D_\circ) \cdot x_0 $.
Then by applying 
Lemma \ref{ISOTROPIE}, with $\mathcal F =   \mathcal D^\llcorner_\circ $ and
$\mathcal C = \conv ( \mathcal D^\llcorner_\circ)$, we obtain  that
 $G_{f(\exp iX\cdot x_0)}=G_{\exp iX\cdot x_0} $, 
for every $X$ in $\conv ( {\mathcal D}_\circ^\llcorner)$.

Next, we check that the extension of $f$ to 
 $\exp i \conv({\mathcal D}_\circ^\llcorner)\cdot
x_0 $ satisfies  the compatibility condition
of Lemma~\ref{EXTENSION}.  
As a consequence of  Lemma \ref{COMPONENTE}, the convexification
$\conv({\mathcal D}_\circ^\llcorner)$  is $W^0$-invariant.
Denote by $N^0$ the preimage of  $W^0$ in  $N_K(\Lambda_r)$ under the canonical projection
$\pi\colon N_K(\Lambda_r)\to W_K(\Lambda_r)$.  Since both  $\Lambda_r$ and  $\conv({\mathcal D}_\circ^\llcorner)$
are $\Ad_{N^0}$-invariant, the domain 
 $L\exp i  \conv({\mathcal D}_\circ^\llcorner)\cdot x_0$ is $N^0$-invariant.
Moreover, the map   $f\colon  L\exp i  {\mathcal D}_\circ^\llcorner\cdot x_0\to \widehat D$ is $N^0$-equivariant  and
so is its
 extension
 to  $L\exp i\conv({\mathcal D}_\circ^\llcorner)
\cdot x_0$, by Proposition \ref{UNIVERSALITY}.
Hence the extension of $f$
to  $\exp i  \conv({\mathcal D}_\circ^\llcorner)\cdot
x_0 $ satisfies  all the assumptions
of Lemma~\ref{EXTENSION} and  $f$ extends to a
holomorphic, $G$-equivariant map
$$G\exp   i\conv({\mathcal D}_\circ^\llcorner)\cdot x_0\to \widehat D,$$
as claimed.
\end{proof}

\bn
{\bf Reduction 2.}
Given  a domain $D=G\exp i\mathcal D^\llcorner\cdot x_0$, the second reduction consists of showing that the map $f\colon D\to \widehat D$  has a $G$-equivariant holomorphic extension to the domain $\widetilde D=G\exp i \widetilde{\mathcal  D}^\llcorner\cdot x_0$, where the set  $\widetilde {\mathcal  D}^\llcorner$ is the convex envelope of ${\mathcal  D}^\llcorner$.  

We first need to recall some properties of the universal covering of the
isometry group of the Euclidean plane.
Namely, let $\widetilde S := \R \ltimes \R^2$ be the semidirect product Lie group
with the multiplication defined by
$$\left (t, \begin{pmatrix} a \cr b \end{pmatrix} \right )\cdot 
\left (t', \begin{pmatrix} a' \cr b' \end{pmatrix} \right ): =
\left (t+t', \begin{pmatrix} \cos t & -\sin t \cr \sin t & \cos t
\end{pmatrix} \begin{pmatrix}a' \cr b' \end{pmatrix} +\begin{pmatrix} a \cr b \end{pmatrix} \right )\,.$$
Its Lie algebra $\s$ is isomorphic to $\R^3$. If  $\{\widetilde L, \, \widetilde M,\, \widetilde N \}$ denotes the canonical basis of $\R^3$, then  the Lie algebra structure is defined by
$$[\widetilde L,\widetilde M]= \widetilde N\,, \quad [\widetilde L,\widetilde N]= - \widetilde M\,, \quad [\widetilde M,\widetilde N]= 0\,.$$
In particular, $\widetilde S$ is a solvable Lie group.

The universal complexification of $\widetilde S$ is  given by $\widetilde S^\C := \C \ltimes \C^2$,
endowed with the extended multiplication law.
Consider the subgroup
$$\widetilde H^\C:= \left \{ \,\left (t+ is, \begin{pmatrix} 0 \cr 0 \end{pmatrix} \right ) \ : \ t+ is \in \C \right \}$$ of $\widetilde S^\C$ with Lie algebra $\C \widetilde L$. In order to perform the second reduction we  embed
$\widetilde S$-invariant subdomains of $\widetilde S^\C/\widetilde H^\C$ into $D$.
We  use the following facts, which can be easily verified.


\bigskip
\begin{lemma}
\label{PONTE}
\item{{\rm (i)}} The map $\C^2 \to \widetilde S^\C/ \widetilde H^\C$,
defined by
$(z,w) \to \left (0, \begin{pmatrix} z \cr w \end{pmatrix} \right )\widetilde H^\C ,$
is a biholomorphism.
\item{{\rm (ii)}} The $\widetilde S$-invariant domains in $\widetilde S^\C/ \widetilde H^\C$ correspond to tube domains in $\C^2$ whose bases are annuli.
\item{{\rm (iii)}} Any such tube domain $\R^2+i\Omega $ is Stein if and only if
the base $\Omega$ is convex, i.e. a disc.
In particular, if $\R^2+i\Omega$ is Stein, then  $\Omega$  contains the origin.
\item{{\rm (iv)}} The orbit of the base point $e\widetilde H^\C$ under the one-parameter subgroup
$\exp i \R \widetilde M$ is a slice for the left
$\widetilde S$-action on $\widetilde S^\C/\widetilde H^\C$. There is a homeomorphism
$$\widetilde S\, \backslash \, \widetilde H^\C/
\widetilde S^\C \cong \R \widetilde M /\Z_2,$$
where the $\Z_2$-action on $\R \widetilde M$ is generated by  the restriction of $\Ad_{\exp\pi \widetilde L}$ to $\R \widetilde M$,  namely the reflection given by 
$\widetilde M \to -\widetilde M$.
\end{lemma}

\bigskip
The crucial step of reduction 2 deals  with the case of two convex connected components of $ {\mathcal  D}^\llcorner$  symmetrically placed with respect to the fixed point set of a reflection $\gamma\in W_K(\Lambda_r)\setminus W^0$.
The action of $\gamma$ decomposes $\Lambda_r$   into the direct sum
$$\Lambda_r= Fix(\gamma)\oplus Fix(\gamma)^\perp .$$
Denote by $Z_G(Fix(\gamma))$ the centralizer  of $Fix(\gamma)$  in $G$, and by $Z_\g(Fix(\gamma))$ its Lie algebra.

\begin{lemma}\label{SOLVABLE} The Lie algebra  $Z_\g(Fix(\gamma))$ contains a 3-dimensional
solvable subalgebra isomorphic to the Lie algebra $\s={\rm Lie}(\widetilde S)$.  
\item{$(i)$}  There exists a Lie group morphism
$\psi: \widetilde S^\C \to G^\C$ mapping $\widetilde H^\C$ to $K^\C$;
\item{$(ii)$}  the group morphism $\psi $ induces a closed embedding 
 $ \widetilde S^\C/\widetilde H^\C  \to G^\C/K^\C\,.$ 
\end{lemma}

\begin{proof} (i) Recall that the restricted root system of $\g$ is
either of type $C_r$  or of type $BC_r$ (see Sect.\,2).
For simplicity of exposition we  assume  $\gamma:=\gamma_{12}$, the reflection flipping the first and the second coordinates (the remaining cases can be dealt in the same way).  Then $Fix(\gamma)^\perp=\R (E_1-E_2)$ and $Fix(\gamma)=\span\{E_1+E_2,E_3,\ldots E_r\}$.
Take an arbitrary element $Q \in \g^{e_1-e_2}$ and set
$$L:=Q +\theta Q, \quad M:=E_1-E_2,\quad N:=[L,M].$$
We first show  that $L,M,N$ lie in the centralizer $Z_\g(Fix(\gamma))$.
By construction, one has that
$$L\in \k ,\quad  M \in \g^{2e_1} \oplus \g^{2e_2},\quad  N\in \g^{e_1+e_2}.$$
In order to see that $[L,E_1+E_2]=0$, let $Z_0={1\over 2} \sum_j T_j +S$, with $T_j=E_j+\theta E_j$ and $S\in Z_\k(\a)$, be the central element in $\k$  given in (\ref{CENTER}).  Since $[L,T_j]= 0$ for $j = 3, \dots, r$, and  the terms  $ [L,T_1 + T_2]$ and  $[L,S]$ are linearly independent, the relation
$ [L,Z_0]=0$ implies  $ [L,T_1 + T_2]=[L,S]=0$.
From  $ [L,T_1 + T_2]=0$ and the identity $\theta L=L$, it follows that
$ [L,E_1 + E_2]+\theta [L,E_1 + E_2]=0.$
This is equivalent to $[L,E_1 + E_2]\in \g^{e_1+e_2}\cap \p$ and implies $[L,E_1 + E_2]=0$, as desired.
The remaining bracket relations
$$[L, E_j]  = [M, E_j] = [N, E_j] = 0, ~\hbox{for}~  j \ge 3,\qquad
[M, E_1 + E_2] = [N,E_1 + E_2] = 0,$$
are all straightforward.

Next we prove that the vectors $\{L,M,N\}$ generate a 3-dimensional solvable subalgebra of $\g$ isomorphic to the algebra $\s:=Lie(\widetilde S)$, discussed
above.
In order to see this, observe that
$ [M,N]=0$. Then it remains to show that, by normalizing $Q$
if necessary, one has  $[L,N] = -M$.
Endow the 3-dimensional subspace  of $\g$
$$V:=\g^{2e_1}\oplus \g^{2e_2}\oplus \R N,$$  with  the restriction of the $\Ad_K$-invariant inner product of $\g$, defined by $B_\theta(X,Y):=-B(X,\theta Y)$, for $X,Y\in\g$.
One can easily verify that the vectors
$\{ E_1+E_2,M=E_1-E_2, N=[L,M]\}$ form an orthogonal basis of $V$ with respect to $B_\theta$.
Since $\ad_L$ is a skew-symmetric operator and  $[L,E_1+E_2]=0$,
the 2-dimensional  subspace $Span\{M,N\} $ is $\ad_L$-stable in $V$.  Thus one can normalize $Q$ so that  $\ad_L(N)=-M$, as desired.

\pn
(ii) Under the identification of $\C^2$ with $ \widetilde S^\C/\widetilde H^\C$  given in Lemma \ref{PONTE},  the induced map is given by $(z,w) \to \exp (zM + wN) \cdot x_0$.
Its image  can be viewed as the orbit
through the base point $x_0$
of the abelian subgroup   with Lie algebra $\span_\C
\{M, \, N \}$. Now the
result follows from  the injectivity of the map $\iota$ defined in (\ref{IOTA}) and  Theorem 2 in \cite{Ros61},
stating that  the orbits of a
unipotent subgroup in the affine space $G^\C/K^\C$ are closed.
\end{proof}


\bn
{\bf Example.} As an example take  $G=Sp(r,\R)$.
Fix
$$Q=  \begin{pmatrix}
\check Q&O&O&O\cr O&O&O&O\cr O&O& -\check Q^t &O\cr O&O&O&O\cr
\end{pmatrix} \in \g^{e_1-e_2},\qquad {\rm with }\quad\check
Q=\begin{pmatrix} 0&-1/2\cr 0&0\end{pmatrix}.$$
The Lie subalgebra of $\g$
generated by the matrices
$$L= \begin{pmatrix}
\check L &O&O&O\cr O&O&O&O\cr O&O& \check L &O\cr O&O&O&O\cr
\end{pmatrix},\quad M=
\begin{pmatrix}O&O&\check M&O\cr O&O&O&O\cr O&O&O&O\cr O&O&O&O\cr \end{pmatrix}
,\quad N= \begin{pmatrix}
O&O& \check N&O\cr O&O&O&O\cr O&O&O&O\cr O&O&O&O\cr
\end{pmatrix},$$
where
$$\check L =
\begin{pmatrix}0&-1/2 \cr 1/2 &0  \end{pmatrix},\quad \check M
= \begin{pmatrix}1&0\cr 0&-1 \end{pmatrix},\quad \check N= \begin{pmatrix}0& 1\cr 1&0 \end{pmatrix},$$
is isomorphic to $\s$.
The  corresponding group
is closed in $Sp(r,\R)$ and
given by
$$
\begin{pmatrix}
U&O&B&O\cr O&I_{r-2}&O&O\cr O&O&U&O\cr O&O&O&I_{r-2}\cr
\end{pmatrix},\qquad U\in SO(2),~  B={}^tB,~tr(B)=0.$$

 \bn
 
\begin{prop}
\label{REDUCTION2}
{\bf (Reduction 2)} Let $\gamma$ be a reflection in $  W_K(\Lambda_r) \setminus W^0$.
The map $$f\colon G\exp i ({\mathcal D}_\circ^\llcorner\cup \gamma \cdot  {\mathcal D}_\circ^\llcorner) \cdot x_0\to \widehat D$$ has a $G$-equivariant, holomorphic extension to the   domain
$$\widetilde D= G \exp i ~ \conv({\mathcal D}_\circ^\llcorner\cup \gamma
 \cdot {\mathcal D}_\circ^\llcorner)\cdot x_0 .$$
\end{prop}

\smallskip
\begin{proof}
Again  for simplicity  of exposition we assume $\gamma=\gamma_{12}$. Then  $Fix(\gamma)=\span\{E_1+E_2,E_3,\ldots, E_r\}$ and $Fix(\gamma)^\perp =\R (E_1-E_2)$. 
Now set $M:=E_1-E_2 $ and let $N$ and $L$ be as in the proof of Lemma \ref{SOLVABLE}. Denote by $\s$ the Lie subalgebra of $Z_\g(Fix(\gamma))$ generated by $\{L,~M,~N\}$ and by $S$ the corresponding subgroup in $Z_G(Fix(\gamma))$.
Denote by $\m$ the abelian subalgebra of $\s$ generated by $\{M,~N\}$, and by 
 $H$ the (possibly non-closed) subgroup of 
$Z_G(Fix(\gamma)) \cap K$ with Lie algebra $\R L$.  

By reduction 1
we  may assume that ${\mathcal D}_\circ^\llcorner$ is convex.
Let $X$ be an arbitrary element in ${\mathcal D}_\circ^\llcorner$. Then $X$  decomposes in a unique
way as  $$X=Y +Z,$$ where
$Y=Y(X) \in Fix(\gamma)$ and $Z=Z(X)\in Fix(\gamma)^\perp =\R M $   depend continuously on~$X$.
For $X\in  {\mathcal D}_\circ^\llcorner $, define 
$$\Sigma_{Y}:=\R M \bigcap ( ({\mathcal D}_\circ^\llcorner \cup
\gamma \cdot  {\mathcal D}_\circ^\llcorner)
-Y)\,,$$ 
and 
$$A_{Y}:=\Ad_H\Sigma_{Y}. $$  Since the Adjoint action of  $H$ on $\m$ is by rotations, the set 
$A_Y$ is an  annulus in $\m$.  
Denote by
\begin{equation}\label{TUBEY} T_Y :=  \exp (iA_{Y} + \m) \cdot x_0=S\exp i{\Sigma_{Y} } \cdot x_0 \end{equation} 
the image of the tube domain $iA_{Y} + \m$ in $ \m^\C \cong \C^2$ under the embedding 
\begin{equation}\label{IOTA2}\iota : \m^\C \to G^\C/K^\C, \quad 
 W \to \exp W \cdot x_0 \end{equation}
(see Lemma \ref{PONTE}(i) and  Lemma \ref{SOLVABLE}(ii)).
Note that $Y + \Sigma_Y$ is contained in $ {\mathcal D}_\circ^\llcorner \cup
\gamma \cdot  {\mathcal D}_\circ^\llcorner$. Since $Y \in Fix(\gamma)$ and $S$ centralizes $ Fix(\gamma) $,  
  right-translation by    $\exp iY$
$$T_Y\to D ,\qquad \exp W\cdot x_0 \mapsto  \exp iY \,\exp W\cdot x_0,\qquad W\in iA_Y+\m  $$
is  $S$-equivariant, and so is the holomorphic map   
$$f_Y:T_Y \to \widehat D,\qquad \exp W \cdot x_0 \to f(\exp iY \exp W\cdot x_0)\,.$$
  Now recall that by Bochner's tube theorem, the  envelope of holomorphy of $T_Y$ is  univalent 
and  given by 
$$\widehat T_Y =\exp (i\conv(A_{Y}) + \m) \cdot x_0= S\exp i\conv(\Sigma_{Y})\cdot x_0 $$
(note that $\conv(\Sigma_{Y})=\conv(A_{Y})\cap \R M$). In particular, it is contained in $\Xi^+$. 
Hence, by Lemma \ref{UNIVERSALITY},  the map $f_Y$  extends  holomorphically
and  $S$-equivariantly to
 $$\hat f_Y: \widehat T_{Y}  \to \widehat D\,.$$
As $X$ varies in ${\mathcal D}_\circ^\llcorner$, one obtains  a
family of $S$-equivariant holomorphic  maps $\hat f_{Y}$,  parametrized by  $Y$.
Set $$\widetilde {\mathcal D}:= \bigcup_{X \in {\mathcal D}_\circ^\llcorner}
Y + \conv ({\Sigma_{Y}})\,. $$
Then an argument similar to the one of  Lemma 7.7 (iv) in \cite{Nee98},
shows that $$\widetilde {\mathcal D}=\conv({\mathcal D}_\circ^\llcorner \cup \gamma \cdot  {\mathcal D}_\circ^\llcorner) .$$
We define a candidate for the desired extension $\hat f\colon \exp i \widetilde {\mathcal D} \cdot x_0 \to \widehat D$ as follows 
\begin{equation} 
\label{FHAT}
 \hat f(\exp iX \cdot x_0):= \hat f_{Y}
(\exp iZ \cdot x_0) \,.
\end{equation}
First of all, the map  $\hat f$ coincides with $f$ on   ${\exp i ({\mathcal D}_\circ^\llcorner \cup \gamma \cdot  {\mathcal D}_\circ^\llcorner) \cdot x_0}$, since for $X \in 
{\mathcal D}_\circ^\llcorner \cup \gamma \cdot  {\mathcal D}_\circ^\llcorner$
one has that $Z\in  \Sigma_Y$ and 
$$ \hat f(\exp iX \cdot x_0)= \hat f_{Y}(\exp iZ \cdot x_0)=f_{Y}(\exp iZ \cdot x_0)=$$
$$= f(\exp iY\exp i Z\cdot x_0) = f(\exp i(Y+ Z) \cdot x_0)=
f(\exp iX \cdot x_0)\,.$$

\sn
In order to apply the extension Lemma \ref{EXTENSION} and obtain  a $G$-equivariant holomorphic extension of $f$ to $G \exp i\widetilde{\mathcal D}\cdot x_0$,  we need to check that $\hat f$ defined in (\ref{FHAT}) meets all the necessary assumptions. 

\sn
$\bullet$  The map  $\hat f$  is a lift  of the natural inclusion  $\exp i\widetilde{{\mathcal D}}\cdot x_0\hookrightarrow \Xi^+$. 
\pn
Since $\hat f$ extends $f$,  one has 
 $q\circ \hat f(\exp iX\cdot x_0)= \exp iX\cdot x_0
\,,$  for all $X\in {\mathcal D}_\circ^\llcorner \cup \gamma
\cdot   {\mathcal D}_\circ^\llcorner\,.$ 
In particular, from (\ref{TUBEY}), the $S$-equivariance of  $q  \circ f_{Y}$ and the fact that  
 $S$ centralizes $Y$,  one has
$$ q\circ f_{Y}(\exp iZ \cdot x_0)= \exp iY \exp iZ \cdot x_0\,,\quad\hbox{for all $Z\in \Sigma_Y$.}$$
By applying the analytic continuation principle to each
$q \circ \hat f_Y: \widehat T_Y \to G^\C/K^\C$, one obtains
$$q \circ \hat f(\exp iX \cdot x_0)= q \circ \hat  f_{Y}(\exp iZ \cdot x_0) =
\exp iY \exp iZ \cdot x_0=  \exp iX \cdot x_0$$
for all $X \in  \widetilde {\mathcal D}$.

\medskip
\noindent
$\bullet$  The map  $\hat f$  is continuous.
\pn
The Stein Riemann domain
$ \widehat  D$ admits
a holomorphic embedding in some $\C^N$. Then, in order to prove that
$\hat f$ is continuous, it is sufficient to show that
given an arbitrary holomorphic function $F: \widehat  D \to \C$, the composition
$ F \circ \hat f:   \exp i \widetilde {\mathcal D} \cdot x_0 \to \C$
is continuous. Since the map $\iota$ in (\ref{IOTA2}) 
is an embedding,  
this is equivalent to checking that the map
$$F \circ \hat f \circ \iota |_{i\widetilde {\mathcal D}}:
i\widetilde {\mathcal D} \to \C, \quad \quad iX \to
F \circ \hat f (\exp iX \cdot x_0)$$
is continuous.

Choose an open  set $U$   in
$Fix(\gamma)$ and an open $\gamma$-invariant subset $\Sigma$ in $\R M=Fix(\gamma)^\perp$
such that $U+ \Sigma \subset  {\mathcal D}_\circ^\llcorner \cup
\gamma \cdot  {\mathcal D}_\circ^\llcorner$. By the definition of $\Sigma$,  for 
$Y \in U$
 the functions $f_{Y}$  are all defined on the
  tube domain $T_\Sigma=S\exp i\Sigma\cdot x_0$.
Moreover, the map
$$U \to \mathcal O(T_{\Sigma}, \C)\,, \quad Y \to F \circ  f_{Y}|_{T_\Sigma}$$
is continuous with respect to  the compact-open topology
on  the Fr\'echet algebra
$\mathcal O(T_{\Sigma}, \C)$ of holomorphic
functions on $T_{\Sigma}$.
Indeed, for $W\in iA_Y+\m$   and $Y \in  U$,  one has
$$F \circ f_{Y}(\exp W \cdot x_0)=F \circ f(\exp iY\exp W \cdot x_0)=
 F \circ f(\exp (iY  + W) \cdot x_0) \,.
$$
Thus,  if $Y_n \to Y_0$, then  $F \circ  f_{Y_n}
\to F \circ  f_{Y_0} $
uniformly on compact subsets of $T_{\Sigma}$.
Since the extension map  $\mathcal O(T_\Sigma, \C) \to \mathcal 
O(\widehat T_\Sigma, \C)$ 
is continuous (see cap.\,I in \cite{Gun90}), 
 it follows that also 
 the map $$U \to \mathcal O(\widehat T_\Sigma, \C)\,, \quad Y \to
F \circ \hat f_Y |_{\widehat T_\Sigma} $$
is  continuous with respect to the compact-open  topology on  
$\mathcal O(\widehat T_\Sigma, \C)$.
As we already remarked, one has   $\widehat T_\Sigma= 
S \exp i \conv (\Sigma) \cdot x_0$.
As a consequence, the map
$$F \circ  \hat f \circ \iota |_{i(U +\conv(\Sigma))}: i(U +\conv(\Sigma)) \to \C\,,$$
defined by
$$iX \to F \circ \hat f (\exp iX \cdot x_0)=  F \circ  \hat f_Y(\exp iZ \cdot x_0)$$
is continuous. Since the domains of the form 
$i(U+ \conv(\Sigma))$ cover $i \widetilde {\mathcal D}$,
the map $\hat f$ is continuous.

\medskip
\noindent
$\bullet$ For all $X\in \widetilde {\mathcal D}$,  one has $G_{\hat f(\exp iX\cdot x_0)}=G_{\exp iX \cdot x_0}$. 
\pn 
We apply Lemma \ref{ISOTROPIE}, with $\mathcal C= \widetilde {\mathcal D}$ and $\mathcal F={\mathcal D}_\circ^\llcorner \cup
\gamma \cdot  {\mathcal D}_\circ^\llcorner$.   
In order to check condition (ii) of the lemma, let $X=Y+ Z \in Y + \conv (\Sigma_Y)$ be an arbitrary element
of $\mathcal C \setminus \mathcal F$. Then there exists $Z' \in \Sigma_Y$ such that $X' = Y+Z' \in 
Y + \Sigma_Y \subset  \mathcal F$ and  the 
one dimensional complex submanifold in (ii)
of Lemma \ref{ISOTROPIE} is given by
$$ \mathcal S:= \{\, \exp i(X' + \lambda (X-X'))\cdot x_0 \ : {\rm Re} \lambda \in [0,1] \, \}=$$
$$= \{\, \exp i(Y + Z' + \lambda (Z-Z'))\cdot x_0 \ : {\rm Re} \lambda \in [0,1] \, \}\,.$$
Note that $Z-Z'$ belongs to $\R M$ and that the strip
$$\{i(Z' + \lambda (Z-Z')))\cdot x_0 \ : {\rm Re} \lambda \in [0,1] \, \}$$
is contained in $i \conv (\Sigma_Y) + \m$.
Thus $\exp i( Z' + \lambda (Z-Z')) \cdot x_0 \in T_Y$ and 
one has a natural holomorphic extension of $ f$ to the 
one dimensional complex submanifold 
$\mathcal S$, namely  
$$\hat f ( \exp i(Y + Z' + \lambda (Z-Z')) \cdot x_0)=
\hat f _Y(\exp i( Z' + \lambda (Z-Z')) \cdot x_0)\,.$$
This shows that  we can apply Lemma 
\ref{ISOTROPIE}, as claimed.

\medskip
\noindent
$\bullet$ The map  $\hat f$   satisfies  the compatibility condition.
\pn
Let  $k_\gamma\in H$ be
the element inducing the reflection with respect to the origin in $\R M$. Since
$H$ centralizes $Fix(\gamma)$, the element   $k_\gamma$ belongs to
$N_K(\Lambda_r^\llcorner)$ and induces the reflection $\gamma$
given in the statement. Hence, 
 for every $X \in \widetilde {\mathcal D}$
  one has $\gamma \cdot X  
=  \Ad_{k_\gamma} X$. Moreover,  by  the  $H$-equivariance of the maps
$\hat f_{Y}$, one obtains the identity
$$\hat f(\exp (i \, \gamma \cdot X) \cdot x_0)=
\hat f(\exp i (Y+\gamma \cdot Z) \cdot x_0)=\hat f_Y(\exp i \Ad_{k_\gamma} Z \cdot x_0)=$$
$$=\hat f_Y(k_\gamma \exp i Z \cdot x_0)=
k_\gamma \cdot \hat f_{Y}(\exp iZ \cdot x_0)=
k_\gamma \cdot \hat f(\exp i X \cdot x_0) \, ,$$
which is the desired compatibility condition.

\medskip
In conclusion, since the map $\hat f$ defined in (\ref{FHAT}) meets all the assumptions of  Lemma
\ref{EXTENSION}, it  extends to a $G$-equivariant, holomorphic  map
$$\hat f \colon G\exp i\conv({\mathcal D}_\circ^\llcorner\cup \gamma
 \cdot {\mathcal D}_\circ^\llcorner)\cdot x_0\to \widehat D\,,$$
as claimed.
\end{proof}

\mn
\begin{cor} By iterating the above reduction 2 finitely many times, we obtain a $G$-equivariant holomorphic  extension of $f\colon D=G\exp i {\mathcal D}^\llcorner \cdot x_0\to\widehat D$ to 
$$\hat f\colon G\exp i\conv(\mathcal D^\llcorner)\cdot x_0\to \widehat D.$$
\end{cor}

\bn

\section{The main theorem}

Let $D$ be a $G$-invariant domain in $\Xi^+$.  Assume that $D$ is  not entirely contained  in the crown $\Xi$ nor in the domain $S^+$ (in the tube case).  In this section we  prove our main theorem, namely   that the envelope of holomorphy  $\widehat D$ of $D$ is univalent and coincides with $\Xi^+$ (Thm. \ref{MAIN}). 
As a by-product we obtain that every Stein $G$-invariant subdomain of $\Xi^+$ is either contained in $\Xi$ or, in the tube case,  in $S^+$  (Thm.\,\ref{SUBDOMAINS}).


\bigskip
\begin{theorem}
\label{MAIN}
  Let $\,G/K\,$ be an irreducible Hermitian symmetric space.
Given a $\,G$-invariant domain  $\,D \,$ in $\,\Xi^+$, denote by $\,\widehat D\,$
its envelope of holomorphy.
\pn
{\rm (i)}   Assume $\,G/K\,$ is of tube type. If $\,D\,$ is not contained in $\,\Xi\,$ nor in
$\,S^+$, then $\,\widehat D\,$ is univalent and coincides with $\,\Xi^+\,$.
\pn
{\rm (ii)}  Assume $\,G/K\,$ is not of tube type.  If $\,D\,$ is not contained in $\,\Xi\,$, then $\,\widehat D\,$ is univalent and coincides with $\,\Xi^+$.

    \end{theorem}

\smallskip
\begin{proof} The proof of the theorem  consists of a sequence of rank-one reductions and convexifications (reduction 1), until an extension of the lift $f \colon \exp i{\mathcal D^\llcorner}\cdot x_0\to \widehat D$ 
to the whole $\exp i \Lambda_r^\llcorner\cdot x_0$ is obtained. 
Such an extension is constructed so that it satisfies the assumptions of  Lemma \ref{EXTENSION} and yields  a $G$-equivariant  holomorphic extension of the map $f\colon D\to \widehat D$  to the whole~$\Xi^+$. 
Then the theorem follows from (ii) of Proposition \ref{TARGET}.
 
\REM{
As in  Section 2,  denote by  $G_j$ (resp. $G_j^\bullet$) the rank-one subgroup (resp. the real rank-one subgroup) of $G$   associated to the root $ \lambda_j\in \{\lambda_1,\ldots,\lambda_r\}$,  
and by $K_j$ (resp. $K_j^\bullet$) the intersection $G_j\cap K$  (resp. $K_j^\bullet\cap G$). Then $G_j/K_j$ and $G_j^\bullet /K_j^\bullet$ are rank-one Hermitian symmetric spaces  with complexification  
$G_j^\C/K_j^\C\cong SL(2,\C)/SO(2,\C)$ and $(G_j^\bullet)^\C/(K_j^\bullet)^\C\cong SL(m+1,\C)/SL(m,\C)$, for some $m\in\N$, respectively. The envelopes of holomorphy of  invariant domains in $G_j^\C/K_j^\C$  and   in  $(G_j^\bullet)^\C/(K_j^\bullet)^\C$ are  univalent and described by Theorem \ref{??} and Theorem \ref{???}, respectively.  }

\sn 
We need to distinguish several cases.

\sn
{\it Case 1.} 
We first consider   a domain $D=G\exp i{\mathcal D}^\llcorner\cdot x_0$ in $ \Xi^+ $   
 satisfying the condition 
$$  {\mathcal D}^\llcorner~\bigcap ~\Lambda_r^{\llcorner}\setminus \left(\bigoplus_{j=1}^r[0,1)E_j\bigcup \bigoplus_{j=1}^r(1,\infty)E_j\right) \not=\emptyset.$$
The above condition implies that  $D$  is not contained in the crown $\Xi$  nor, in the tube case,  in the  domain  $S^+$.
By reductions 1 and 2, we can assume that  ${\mathcal D}^\llcorner$ is  a $W_K(\Lambda_r)$-invariant, open 
 convex subset of $\Lambda_r^{\llcorner}$.  
We claim that   ${\mathcal D}^\llcorner$ contains a point $X$   
with exactly one coordinate equal to 1, and the other ones either all $<1$ (Case 1.a) or all $>1$ (Case 1.b).  
This follows from the fact that an open $W_K(\Lambda_r)$-invariant convex set in $\Lambda_r^\llcorner$ intersects 
at least one of the open sets  $$\bigoplus_j[0,1)E_j\quad \hbox{or}\quad \bigoplus_j(1,+\infty)E_j.$$ In particular, it contains an open piece of  its boundary. Since the points with exactly one coordinate equal to 1 form an open dense subset of the boundary of each  set, the claim follows.

For the rank-one reduction, denote by   $G_j$  the rank-one subgroup  of $G$  associated to the root $ \lambda_j\in \{\lambda_1,\ldots,\lambda_r\}$,  
and by $K_j$  the intersection $G_j\cap K$ (see Sect.2). The quotient $G_j/K_j$ is a  rank-one Hermitian symmetric space of tube-type.  
The envelope  of holomorphy of an  invariant domain  in $G_j^\C/K_j^\C$ is  univalent and described by Theorem \ref{RANKONE1}.

\bn
{\it Case 1.a.}   By the  $W_K(\Lambda_r)$-invariance of ${\mathcal D}^\llcorner$ ,   we can assume that $({\mathcal D}^\llcorner)^+$ contains a point
\begin{equation} \label{HP1}X=(1,x_2, \dots, x_r),\qquad\hbox{with}\quad  1>x_2>\ldots x_r>0.\end{equation}   

\pn
Our first goal is to obtain an extension of  $f$ to a set 
$\exp i\widetilde{\mathcal D}\cdot x_0$, where
 $\widetilde{\mathcal D}$ is an open   $W_K(\Lambda_r)$-invariant convex set in $\Lambda_r^\llcorner$ containing $\mathcal D^\llcorner$ and  the point $(1,0,\ldots,0).$
This requires a number of steps aimed at  gradually extending  $f$  to  $W_K(\Lambda_r)$-invariant larger  sets
 $\exp i \mathcal C\cdot x_0$, with $\mathcal C$ containing $\mathcal D^\llcorner$ and, in order,  the points  
$$ (1,x_2,\ldots,x_{r-2},x_{r-1},0),\quad  (1,x_2,\ldots,x_{r-2},0,0) ,~\ldots ~, ~(1,0,\ldots,0) .$$  
 Denote by  $$\int(\overline{(\Lambda_r^\llcorner)^+})$$  the interior  of $\overline{(\Lambda_r^\llcorner)^+}$ (in the relative topology), which coincides with $\overline{(\Lambda_r^\llcorner)^+}\setminus {\mathcal H},$ 
where $\mathcal H:= \cup_{\gamma \in W_K(\Lambda_r)} \{Fix(\gamma)\}$ 
denotes the set of reflection hyperplanes in $\Lambda_r^\llcorner$.   Likewise, denote by 
$$\int( ({\mathcal D}^\llcorner)^+ )=({\mathcal D}^\llcorner)^+\cap \int(\overline{(\Lambda_r^\llcorner)^+})$$ the interior  of $({\mathcal D}^\llcorner)^+$ in the relative topology.  
 Under the above assumption (\ref{HP1}),  the interior of $({\mathcal D}^\llcorner)^+$  contains  an  open set  of the form
$$U+V,\quad \hbox{with} \quad   (1,x_2, \dots x_{r-1},0)\in U\subset E_r^\perp,\quad{\rm and}\quad V=(a_r,b_r) E_r .$$
 \pn
Write an arbitrary element $W\in U+V$ as 
$$W=Y+Z,\quad \hbox{with} \quad Y=Y(W)\in U ~{\rm and}~Z=Z(W)\in V $$
depending continuously on $W$.
Define  
$$D_r:=G_r\exp i V\cdot x_0.$$ 
Since $G_r$ commutes with the subgroups $G_j$, for $j\not=r$,  right    translation by $\exp iY$
$$D_r\to D,\quad g\exp iZ\cdot x_0\mapsto \exp iY g\exp iZ\cdot x_0$$
is $G_r$-equivariant and  so is the holomorphic map  
$$f_Y\colon D_r\to \widehat D, \quad g\exp iZ\cdot x_0\to f(\exp iYg\exp iZ\cdot x_0).$$
Recall that the envelope of holomorphy of $D_r$ is univalent and given by $\widehat D_r= G_r\exp i \widehat V\cdot x_0$, with $\widehat V=[0,b_r)E_r$ (see Prop. \ref{RANKONE1}(i)).  
As $W$ varies in $U+V$, one obtains   a family of $G_r$-equivariant holomorphic maps 
$$\hat f_Y\colon \widehat D_r \longrightarrow  \widehat D,$$ parametrized by $Y=Y(W)\in U$. 

Define a map $$\hat f\colon \exp i  (U+\widehat V)\cdot x_0 \to \widehat D,\quad\hbox{by}\quad 
 \exp iW\cdot x_0\to \hat f_Y(\exp iZ\cdot x_0).$$
Observe that $U+\widehat V$  is an open set in~$\Lambda_r^\llcorner $, since  it  is  entirely  contained in the interior of $\overline{(\Lambda_r^\llcorner)^+}$. 
Next we  show that $\hat f$  satisfies all the assumptions of the extension Lemma~\ref{EXTENSION}. 

\sn
- $\hat f$  concides with $f$ on the   set $\exp i (U+ V) \cdot x_0$, since for every $Y\in U$,  
 $$\hat f_Y(\exp iZ\cdot x_0)=f_Y(\exp iZ\cdot x_0), \quad \hbox{for all }\quad Z\in V .$$ 
- $\hat f$ is a lift of the natural inclusion of $\exp i(U+\widehat V)\cdot x_0$ into $\Xi^+$.  The analytic continuation principle, 
applied to each holomorphic map $q\circ \hat f_Y\colon \widehat D_r\to \Xi^+$,  implies 
\begin{equation} \label{RESTRICTION2} q\circ \hat f_{Y}|_{\exp i\widehat V\cdot x_0}=Id|_{\exp i\widehat V\cdot x_0},\qquad   W\in U+\widehat V .
\end{equation}
-  $\hat f$ is continuous, by  similar arguments as the ones used  in Lemma \ref{PONTE}.\pn
-  for every $X\in  U+\widehat V$, one has   
$G_{\exp i X\cdot x_0}=G_{\hat f(\exp i X\cdot x_0)}$. This  follows from Lemma~\ref{ISOTROPIE}, by taking $\mathcal C=U+\widehat V$ and $\mathcal F=U+\widehat V$, and arguing as  in Lemma  \ref{PONTE}.
\pn
- since  $U+\widehat V$ is entirely contained in the perfect slice $\overline{(\Lambda_r^\llcorner)^+} $, the  compatibility conditions of Lemma \ref{EXTENSION} are automatically satisfied.

As a consequence $\hat f\colon {\exp i(U+\widehat V)}\cdot x_0\to \widehat D$  extends to a $G$-equivariant holomorphic map 
\begin{equation}\label{EXTENDED1}\hat f\colon G \exp i(U+\widehat V)\cdot x_0 \to \widehat D.\end{equation}
Note that $G  \exp i(U+\widehat V)\cdot x_0$ is an open set  in $\Xi^+$, which  has open intersection with $D$ and  coincides with $G \exp i\left(W_K(\Lambda_r)\cdot (U+\widehat V)\right) \cdot x_0$. 
By  the analytic continuation principle, the  map (\ref{EXTENDED1}) coincides with  $f\colon D\to \widehat D$  on the points of  $D$. As a result, we have obtained a $G$-equivariant holomorphic  extension of  $f$ to the  larger  domain 
$G\exp i\widetilde{\mathcal D}\cdot x_0$, where
$$\widetilde{\mathcal D}=\mathcal D^\llcorner~ \bigcup ~W_K(\Lambda_r)\cdot (U+\widehat V).$$
The set $\widetilde{\mathcal D}$ contains the point $(1,x_2, \dots x_{r-1},0)$, the projection of the initial point $X$ onto the hyperplane $x_r=0$, and 
by  reduction 1,  may be assumed to be  convex. 
 
So set $\mathcal D^\llcorner=\widetilde{\mathcal D}$ and, for the second step, take  
 an open subset of $\int((\mathcal D^\llcorner)^+)$ of the form $U+V$, with $$ (1,x_2, \dots x_{r-2},0,0)\in  U\subset E_{r-1}^\perp,\quad{\rm and}\quad V=(a_{r-1},b_{r-1})E_{r-1}.$$ 
Write  an arbitrary element $W\in U+V$ as 
$$W=Y+Z,\quad \hbox{with} \quad Y=Y(W)\in U\quad{\rm and}\quad Z=Z(W)\in V $$
depending continuously on $W$.

Define $D_{r-1}:=G_{r-1}\exp i V\cdot x_0$.   
The map 
$$ D_{r-1}\to D,\qquad g\exp iZ\cdot x_0\to  \exp iY g\exp iZ\cdot x_0$$
defines a $G_{r-1}$-equivariant holomorphic embedding of $D_{r-1}$  into $D$ and  induces a $G_{r-1}$-equivariant holomorphic map  
 $ \widehat D_{r-1} \longrightarrow  \widehat D .$ 
As in the previous step, consider the family of $G_{r-1}$-equivariant holomorphic maps 
$$\hat f_Y\colon \widehat D_{r-1} \longrightarrow  \widehat D , $$  
obtained by letting $W$ vary in $ U+V$, and   
define a map $$\hat f\colon \exp i  (U+\widehat V)\cdot x_0 \to \widehat D,\quad{\rm by}\quad   \hat f(\exp iW\cdot x_0):=   \hat f_Y(\exp iZ\cdot x_0). $$
Recall that 
the envelope of holomorphy of $D_{r-1}$ is univalent and given by $\widehat D_{r-1}=G_{r-1}\exp i \widehat V\cdot x_0 $, where $\widehat V:=[0,b_{r-1})E_{r-1}$ (see Prop. \ref{RANKONE1}(i)).  
Since this time $U+\widehat V$ is not entirely contained in the interior of $\overline{(\Lambda_r^\llcorner)^+}$, 
we restrict $\hat f$  to the set $\exp i\mathcal C\cdot x_0$, where 
\begin{equation} \label{INTERSECTION} \mathcal C=(U+\widehat V)\bigcap \int(\overline{(\Lambda_r^\llcorner)^+}) .\end{equation}
The same arguments used in the previous step show that the restricted map $\hat f|_{\exp i\mathcal C\cdot x_0}$, coincides with  $f$ on  $D\cap \exp i\mathcal C\cdot x_0$, it is a lift  of the natural inclusion of $\exp i\mathcal C \cdot x_0$ into $\Xi^+$, and it is continuous.

Moreover, because of (\ref{INTERSECTION}), the compatibility conditions of Lemma \ref{EXTENSION} are automatically  satisfied on $\exp i{\mathcal C}\cdot x_0$. As a consequence,  
 $\hat f|_{\exp i{\mathcal C}\cdot x_0}$ extends to a $G$-equivariant holomorphic map 
$$\hat f\colon G\cdot \exp i{\mathcal C}\cdot x_0 \to \widehat D.$$
By   the analytic continuation principle, the above map coincides with  $f\colon D\to \widehat D$ on the points of $D $. 
As a result,  we have obtained an extension of  $f$ to the domain 
$G\exp i\widetilde{\mathcal D}\cdot x_0$, where
$$\widetilde{\mathcal D}=\mathcal D^\llcorner~ \bigcup ~W_K(\Lambda_r)\cdot (U+\widehat V).$$
The above set  contains the point $(1,x_2, \dots x_{r-2},0,0)$, the projection of the initial point $X$ onto the linear subspace   $x_r=x_{r-1}=0$, and,  
by  reduction 1, it  may be assumed to be  convex. 
 By applying  the  procedure, used for the  coordinate $x_{r-1}$, to  the  coordinates 
$x_{r-2},\ldots x_2$, we obtain  an extension of  $f$ to a domain 
$G\exp i\widetilde{\mathcal D}\cdot x_0$, where
 $\widetilde{\mathcal D}$ is an open,  $W_K(\Lambda_r)$-invariant convex set in $\Lambda_r^\llcorner$ containing $\mathcal D^\llcorner$ and  the point~$(1,0,\ldots,0).$

 Set $\mathcal D^\llcorner=\widetilde{\mathcal  D}$  and, for the
 final step, take an open subset of $\int((\mathcal D^\llcorner)^+)$ of the form $U+V$, with $$U \subset E_1^\perp, ~{\rm and}\quad V=(a_1,b_1)E_1,\quad  a_1<1<b_1  .$$  
This time  $D_1=G_1\exp i V\cdot x_0 $ is a $G_1$-invariant complex submanifold $G^\C/K^\C$ whose envelope of holomorphy is univalent and given by $\widehat D_1=G_1\exp i \widehat V\cdot x_0$, with $\widehat V:=[0,\infty)E_1$ (see Prop. \ref{RANKONE1}(iii)). 

The usual procedure yields a $G$-equivariant extension of  $f\colon D\to\widehat D$  to the whole $\Xi^+\setminus G/K$. Since the orbit  $G/K$ is  a totally real submanifold of $\Xi^+$  (of maximal dimension), $f$ extends to the whole $\Xi^+$ (see \cite{Fie82}), as desired.  

 \bn
{\it Case 1.b.}   By the  $W_K(\Lambda_r)$-invariance of ${\mathcal D}^\llcorner$,   we can assume that $({\mathcal D}^\llcorner)^+$ contains a point
\begin{equation}\label{HP2}X=(x_1,x_2, \dots, 1),\qquad\hbox{with}\quad  x_1>x_2>\ldots x_{r-1}>1.\end{equation}   
Our goal is to contruct  an extension of $f$
 to a set 
$\exp i\widetilde{\mathcal D}\cdot x_0$, where
 $\widetilde{\mathcal D}$ is an open  $W_K(\Lambda_r)$-invariant convex set in $\Lambda_r^\llcorner$ containing $\mathcal D^\llcorner$ and  the point $(1,0,\ldots,0).$ Then the result follows from the last step in Case 1.a. 

 By the above assumption (\ref{HP2}),   the interior of  $({\mathcal D}^\llcorner)^+$  contains  an  open set  of the form
$$U+V,\quad \hbox{with} \quad  (x_1,x_2, \dots x_{r-1},0)\in U\subset E_r^\perp,\quad{\rm and}\quad   V=(a_r,b_r) E_r ,~  a_r<1<b_r .$$
Consider  $D_r=G_r\exp i V  \cdot x_0$ and recall that the envelope of holomorphy of $D_r$ is univalent and given by  $\widehat D_r=G_r\exp i \widehat V\cdot x_0$, with $\widehat V=[0,\infty)E_r$ (see Prop. \ref{RANKONE1}(iii)). The usual procedure
yields an extension of $f$ to a set $G\exp i\widetilde{\mathcal  D}\cdot x_0$, where  
$\widetilde{\mathcal  D}$ is an  open  $W_K(\Lambda_r)$-invariant convex   set in $\Lambda_r^\llcorner$,  containing  $(x_1,x_2, \dots x_{r-1},0)$. 

We claim that $\widetilde{\mathcal  D}$ can be assumed to contain  $(x_1,\dots,x_{r-2}, 1,0)$. 
Since $\widetilde{\mathcal  D}$ is $W_K(\Lambda_r)$-invariant and convex, it contains the point $(x_1,x_2, \dots,0, x_{r-1})$ and  the segment 
 $$(x_1,x_2, \dots ,tx_{r-1},(1-t)x_{r-1}), \quad {\rm for~}t\in[0,1].$$ In particular it contains  the point  $(x_1,x_2, \dots ,2/3\, x_{r-1},1/3\,  x_{r-1})$, which lies in $(\mathcal D^\llcorner)^+$ and has a smaller $(r-1)^{th}$ coordinate than $(x_1,x_2, \dots x_{r-1},0)$.  
  By iterating the procedure, we obtain a convex set $\widetilde{\mathcal  D}$ containing  $(x_1,x_2, \dots ,x_{r-1}^\prime,0)$, for some $x^\prime_{r-1}<1$. Then the claim follows from the convexity of $\widetilde{\mathcal  D}$ and the inequality $ x_{r-1}>1$.  

Acting in the same way on the coordinates $x_{r-1}, x_{r-2},\ldots, x_1$ we obtain a convex set $\mathcal D^\llcorner$ containing $(1,0,\ldots,0)$, as desired.

 \bn
 {\it Case 2.}  It remains to consider the  case of  $G/K$ not of tube-type and   $D=G\exp i{\mathcal D}^\llcorner \cdot x_0$ a domain entirely  contained in $\Omega^+$.  
By  Proposition \ref{SOTTODOMINI}, this is equivalent to requiring    
\begin{equation} \label{HP3} {\mathcal D}^\llcorner \subset  \bigoplus_{j=1}^r(1,\infty)E_j  .\end{equation}
By the  $W_K(\Lambda_r)$-invariance of ${\mathcal D}^\llcorner$ and (\ref{DPERFECT}),   we can assume that $({\mathcal D}^\llcorner)^+$ contains a point  
$$X=(x_1,x_2, \dots x_r),\qquad{\rm with}\quad   x_1> x_2>\ldots> x_r>1.$$
Our goal is to  show  that  the map $f$ extends  to  a set  $ \exp i\widetilde{\mathcal D} \cdot x_0$, where $\widetilde{\mathcal D} $  is an open  $W_K(\Lambda_r)$-invariant convex set in $\Lambda_r^\llcorner$ 
containing ${\mathcal D}^\llcorner$ and  the point $E_1=(1,0,\ldots,0)$.  Then the result follows from the last step in Case 1.a.

By the above assumption (\ref{HP3}), the interior of  $(\mathcal D^\llcorner)^+$  contains   an  open set  of the form
$$U+V,\quad \hbox{with} \quad  U\subset E_r^\perp,~ (x_1,x_2, \dots x_{r-1},0)\in U,\quad V=(a_r,b_r) E_r,\quad a_r>1 .$$
Since we are in the non-tube case,   
we can consider the $\theta$-stable  real rank-one subgroup $ G_r^\bullet$ of $G$  associated to the root $\lambda_r$ (see Sect.\,2).  Then  the intersection $K_r^\bullet:= G_r^\bullet\cap K$ is a maximal subgroup of  $G_r^\bullet$ and 
the quotient $G_j^\bullet/K_j^\bullet$ is a  rank-one Hermitian symmetric space,  not of tube-type. 

Define  $D_r:= G_r^\bullet\exp i V\cdot x_0$.  Since $ G_r^\bullet$ commutes with the rank-one subgroups $G_i$ associated to the roots $\lambda_j$, for $j\not= r$ (Lemma \ref{COMMUTATIVITY}), the map 
$$ D_r\to D,\quad g\exp iZ\cdot x_0 \longrightarrow \exp iY g\exp i Z\cdot x_0$$  
defines a $ G_r^\bullet$-equivariant holomorphic embedding  and  induces a $ G_r^\bullet$-equivariant holomorphic map  
 $ \widehat D_r \longrightarrow  \widehat D .$ 
 Recall that the envelope of holomorphy of $D_r$ is univalent and given by $\widehat D_r= G_r^\bullet\exp i \widehat V\cdot x_0$, with $\widehat V=[0,\infty)E_r$ (see Prop. \ref{RANKONE2}(ii)). 
The same procedure used   in the previous cases, yields a $G$-equivariant holomorphic map 
$$\hat f\colon G\exp i\mathcal C\cdot x_0\to \widehat D,$$
where $\mathcal C$ is an open convex $W_K(\Lambda_r)$-invariant subset of $\Lambda_r^\llcorner$ containing $(x_1,\ldots, x_{r-1},0)$.
By applying the same arguments to the remaining variables $x_{r-1}, x_{r-2},\ldots,x_2$, we achieve the desired extension.  This concludes the proof of the theorem.
\end{proof}

\bigskip

Stein $G$-invariant domains in $\Xi$ and $S^+$ were classified in \cite{GiKr02} and \cite{Nee99}, respectively.  Inside  the crown $\,\Xi$,
as well as inside
$\,S^+\,,$ an invariant domain  can be described via a semisimple abelian  slice,
  and Steiness
is characterized by logarithmic  convexity of such a  slice. These results together with  the above theorem  conclude   the classification of Stein $G$-invariant domains in $\Xi^+$.

\mn
\begin{cor} \label{SUBDOMAINS} Let $G/K$ be a Hermitian symmetric space and let $D$ be a Stein $G$-invariant proper subdomain of $\Xi^+$.  
  \item{{\rm (i)}} If $G/K$ is of tube type, then  either  $D\subseteq \Xi$  or  $D\subseteq S^+$.  
  \item{{\rm (ii)}} If $G/K$ is not of tube type, then $D\subseteq \Xi$.

\end{cor}

\mn
\begin{remark} Let $G/K$ be an arbitrary irreducible, non-compact, Riemannian symmetric space. The crown domain in $G^\C/K^\C$  is given by $\Xi=G\exp i
\Omega_{AG}\cdot x_0$, where $\Omega_{AG}:=\{ H\in\a \ : \ |\alpha(H)|<{\pi\over 2}, ~
{\rm \  for \  all\ }  \alpha \in \Delta(\g,\a) \} $.  Invariant   domains  in $\Xi$ can be written as  $D=G\exp i\Omega\cdot x_0$, for some $W_K(\a)$-invariant open set $\Omega\subset \Omega_{AG}$.
Stein invariant domains have been characterized in \cite{GiKr02} as the ones  for which the slice
 $\exp i \Omega \cdot x_0$ is logarithmically convex.  However, we are not aware of an explicit  univalence
statement  for the envelope of holomorphy of an arbitrary  invariant domain $D\subset \Xi$. For the sake of completeness, we outline the proof of this fact here.

Let   $D=G\exp i\Omega\cdot x_0$ be an invariant domain in $\Xi$.
As $\Xi$ is Stein, one has a commutative diagram
\begin{equation}
\xymatrix{&\widehat D\ar[d]^q\\
D\ar[ur]^f \ar[r]^{Id}& \Xi \ .  }
\label{envelope}
\end{equation}
Define $\Omega^+:=\Omega\cap \overline{\a^+}$, where $\overline{\a^+}$ is a closed Weyl chamber in $\a$, and  let $\Omega_\circ$ be the connected component of
$\Omega$ containing $\Omega^+$. Consider   the set
$\Gamma^0$ of simple reflections in $\a$
whose fixed point hyperplanes contain a non-zero element of
$\overline{\a^+}$ and let  $W^0$ be the subgroup of $W_K(\a)$ generated by $\Gamma^0$. As in Lemma \ref{COMPONENTE}, one can show that~$W^0\cdot \Omega^+=\Omega_\circ $.

Set $A:=\exp \a$ and consider   the $r$-dimensional complex submanifold $A\exp i\Omega_\circ\cdot x_0$ of $D$, which is biholomorphic to a tube domain in $\C^r$ with base $\Omega_\circ$.
The restriction
$f|_{A\exp i\Omega_\circ\cdot x_0} : A\exp i\Omega_\circ\cdot x_0\to \widehat D$
of $f$ to $A\exp i\Omega_\circ\cdot x_0$
extends to an  $A$-equivariant holomorphic map
$A\exp i\conv(\Omega_\circ)\cdot x_0\to \widehat D$.
Then the same arguments  as in   Proposition \ref{REDUCTION1}  show that the inclusion $f\colon D\to\widehat D$ admits a $G$-equivariant extension to the domain $G\exp i\conv(\Omega_\circ)\cdot x_0$.
Thus, without loss of generality, we may assume that all connected components
of $\Omega$ are convex.

The second  part of the proof consists of showing that the map $f$ admits a $G$-equivariant holomorphic extension  to the domain $  G\exp i \conv(\Omega)\cdot x_0$. For this purpose, we first consider the case where  $\Omega$ consists of two connected components $\Omega_\circ$ and~$s_\alpha\cdot \Omega_\circ$, simmetrically placed with respect to
the fixed point hyperplane $Fix(s_\alpha)$ of  a reflection  $s_\alpha\in  W_K(\a)\setminus W^0$.
Let   $H_\alpha$ be a generator of $Fix(s_\alpha)^\perp$.
Choose $X_\alpha\in\g^\alpha$ so that the vectors $\{X_\alpha,\theta X_\alpha, H_\alpha\}$ 
generate a  $\theta$-stable  $\s\l(2)$-subalgebra.  
Denote by $G_\alpha$ the corresponding subgroup of $G$ and by $K_\alpha=G_\alpha \cap K$. The  quotient $G_\alpha/K_\alpha$ is a Hermitian rank-one symmetric space of tube type.

Let $X$ be an arbitrary element in $\Omega_\circ$. Then $X$ decomposes in a unique way as
 $X=Y+Z,$ where $Y=Y(X)\in Fix(s_\alpha)$ and $Z=Z(X)\in \R H_\alpha$ depend continuosly
on $X$. Define
 $$\Sigma_Y :=\R H_\alpha\cap ((\Omega_\circ \cup
s_\alpha\cdot \Omega_\circ)-Y) \quad \hbox{and}\quad
D_Y=G_\alpha\exp i \Sigma_Y \cdot x_0.$$ 
Then $D_Y$ is biholomorphic to a $G_\alpha$-invariant domain inside  the crown $\Xi_\alpha\subset G^\C_\alpha/ K^\C_\alpha$.
By Proposition \ref{RANKONE1}, the  envelope of holomorphy of $D_Y$ is univalent and  it is  given by
$\widehat D_Y=G_\alpha\exp i\conv (\Sigma_Y) \cdot x_0$.

Note that $Y + \Sigma_Y \subset \Omega_\circ \cup
s_\alpha\cdot \Omega_\circ$, and that $\alpha(Y)=0$, for all 
 $Y\in Fix(s_\alpha)$.  It follows that   $\exp iY$ commutes with $G_\alpha$ and the map 
$$D_Y  \to  D,
\quad \quad g \exp iZ \cdot x_0 \to  \exp iY g \exp iZ \cdot x_0\, $$
 is a $G_\alpha$-equivariant embedding.

As $X$ varies in $\Omega_\circ$, one obtains a family of
$G_\alpha$-equivariant holomorphic maps
$f_Y\colon D_Y \to \widehat D,$ defined by
$g\exp iZ\cdot x_0\mapsto f(\exp iYg\exp iZ\cdot x_0)$,  
equivariantly extending to
$\hat f_Y\colon \widehat D_Y\to \widehat D$
(cf. Lemma \ref{UNIVERSALITY}).

Define $\widetilde \Omega:=
\bigcup_{X\in \Omega_\circ}Y + \conv (\Sigma_Y)$ and
 $$
\hat f: \exp i  \widetilde \Omega\cdot x_0\to \widehat D,\quad \exp iX\cdot x_0 \to
\hat f_Y(\exp iZ\cdot x_0).$$
 By Lemma 7.7 in \cite{Nee98}, the set
$\widetilde \Omega $ coincides with $\conv(\Omega_\circ\cup s_\alpha \Omega_\circ).$
Arguments analogous to the ones  used in the proof of  Proposition \ref{REDUCTION2}  show that $\hat f$
is a continuos extension of the lift $f|_{\exp i\Omega\cdot x_0}
\colon\exp i\Omega\cdot x_0\to\widehat D$
and that it satisfies 
the assumptions of Lemma \ref{EXTENSION}.  As a consequence, $\hat f$  further extends 
to a $G$-equivariant holomorphic map
$\hat f\colon G\exp i\widetilde \Omega \cdot x_0\to \widehat D$. 
By iterating  the above procedure if necessary, one eventually obtains a $G$-equivariant holomorphic extension 
$$\hat f\colon G\exp i\conv(\Omega) \cdot x_0\to \widehat D.$$
 Since the domain $G\exp i \conv(\Omega)\cdot x_0$ is Stein (see \cite{GiKr02}), it coincides with the envelope
of holomorphy $\widehat D$ of $D$.
This shows that
the envelope of holomorphy of $D=G\exp i  \Omega\cdot x_0$ is univalent and
given by $ G\exp i \conv(\Omega)\cdot x_0.$
\end{remark}

\bn
\begin{remark}  
The  univalence  result  of Theorem \ref{MAIN} does not hold true
for equivariant Riemann domains $p\colon X\to G^\C/K^\C$, which are not envelopes  of holomorphy. For a Hermitian symmetric space $G/K$ of tube type one can  construct a non-trivial $G$-equivariant Stein covering  of the domain~$S^+\subset \Xi^+$.
\end{remark}


\bn


\begin{thebibliography}{HiOl97}
\bibliographystyle{alpha}
\frenchspacing 

\medskip
\bibitem[AkGi90]{AkGi90}{\sc Akhiezer  D. N., Gindikin  S. G.}
{\it On Stein extensions of real symmetric
spaces.} Math. Ann. {\bf 286} (1990)  1--12.

\medskip
\bibitem[BrTD85]{BrTD85}{\sc Br\"ocker T., tom Dieck T.}
{\it Representations of Compact Lie Groups.} 
G.T.M. 98, Springer-Verlag, Berlin, 1985.


\medskip
\bibitem[Fie82]{Fie82} {\sc Mike J. Field} {\it  Several complex variables and complex manifolds  I.}  London Mathematical Society Lecture Note Series {\bf vol. 65}, Cambridge University Press, Cambridge, 1982.

 \medskip
 \bibitem[Gea02]{Gea02}
 {\sc Geatti L.}
 {\it Invariant domains  in the complexification of a noncompact Riemannian symmetric space.}
 J. of Alg. {\bf 251} (2002) 619--685. 


\medskip
\bibitem[GeIa08]{GeIa08}
{\sc Geatti L., Iannuzzi A.}
{\it Univalence of equivariant Riemann domains over the complexifications of rank-1 Riemannian symmetric spaces.} 
Pac. J. Math., {\bf 238} N.2 (2008) 275--205.

\medskip
\bibitem[GeIa13]{GeIa13}
{\sc Geatti L., Iannuzzi A.}
{\it Orbit structure of a distinguished invariant domain in the complexification  of Hermitian  symmetric space.} 
Preprint  2013.



\medskip
\bibitem[GiKr02]{GiKr02}
{\sc Gindikin S., Kr\"otz B.}
{\it Invariant Stein Domains in Stein Symmetric Spaces 
and a Nonlinear Complex Convexity Theorem.} 
Inter. Math. Res. Not. {\bf 18} (2002) 959--971. 
 
 \medskip
\bibitem[Gun90]{Gun90}
{\sc Gunning R. C.} {\it Introduction to holomorphic functions of several variables: Vol.\,1.} Wadsworth \& Brooks/Cole, 1990.
 
 
\medskip
\bibitem[Kna04]{Kna04}
{\sc Knapp A. W.}
{\it Lie groups beyond an introduction.}
Birkh\"auser, Boston, 2004.

\medskip
 \bibitem[Kro08]{Kro08}
 {\sc Kr\"otz B.}
 {\it Domains of holomorphy for irreducible unitary representations of simple Lie groups.} 
 Inv. Math.  {\bf 172} (2008) 277--288.



 \medskip
 \bibitem[KrOp08]{KrOp08}
 {\sc Kr\"otz B.,  Opdam  E.}
 {\it Analysis on the crown domain.} GAFA, Geom. 
Funct. Anal. {\bf 18} (2008) 1326--1421.

 \medskip 
 \bibitem[KrSt04] {KrSt04}
 {\sc Kr\"otz B., Stanton R.} 
 {\it Holomorphic extensions of representations I: automorphic forms.}
 Ann. of Math.,  {\bf 159} (2004) 641--724.

\medskip 
\bibitem[KrSt05] {KrSt05}
 {\sc Kr\"otz B., Stanton R.} 
 {\it Holomorphic extensions of representations II: geometry and harmonic analysis.}
 GAFA, Geom.  Funct. Anal. {\bf 15} (2005) 190-245.


\medskip 
\bibitem[Moo64] {Moo64}
{\sc Moore C.C.} 
{\it Compactifications of symmetric spaces II. The Cartan domains.}
Amer. J. Math. {\bf 86} (1964) 358-378.


 \medskip
 \bibitem[Nee98]{Nee98}
 {\sc Neeb K.H.}
 {\it On the complex and convex geometry of Ol'shanskii semigroups.}  Ann. Inst. Fourier,
 Grenoble, {\bf 48} (1998)  149--203.  

\medskip
\bibitem[Nee99]{Nee99}
{\sc Neeb K.H.}
{\it On the complex geometry of invariant domains in 
complexified symmetric spaces.}  Ann. Inst. Fourier,
Grenoble, {\bf 49} (1999)  177--225.  


 \medskip
 \bibitem[Ros63]{Ros63}
 {\sc Rossi H.} 
 {\it On envelopes of holomorphy.}
 Comm.  pure and applied Mathematics {\bf 16} (1963) 9--17.

\medskip
\bibitem[Ros61]{Ros61}
{\sc Rosenlicht M.} 
{\it On quotient varieties and the affine embedding of certain homogeneous spaces.}
Trans. Amer. Math. Soc. 101 (1961) 211--223.


\bbigskip

\end{thebibliography}
 \end{document}